\documentclass{amsart}

\usepackage{amsmath,amssymb,,xy,mathrsfs,amsthm,xcolor,amsxtra}
\usepackage{hyperref}
\hypersetup{colorlinks=true,urlcolor=blue,citecolor=blue,linkcolor=blue}
\usepackage{enumerate}
\usepackage{colonequals}
\usepackage{comment}

\usepackage[mathcal]{euscript}
\newcommand{\Gal}{\operatorname{Gal}}
\newcommand{\End}{\operatorname{End}}
\newcommand{\p}{\mathfrak{p}}
\newcommand{\Fp}{\mathbf{F}_{\mathfrak{p}}}
\newcommand{\OK}{\mathcal{O}_{K}}

\newcommand{\Fl}{\mathbf{F}_{\ell}}	

\newcommand{\Aut}{\operatorname{Aut}}

\newcommand{\Zl}{\mathbf{Z}_{\ell}}
\newcommand{\im}{\operatorname{im}}

\newcommand{\Z}{\mathbf{Z}}
\newcommand{\F}{\mathbf{F}}

\newcommand{\Q}{\mathbf{Q}}

\newcommand{\GSp}{\operatorname{GSp}}

\newcommand{\Sp}{\operatorname{Sp}}

\newcommand{\SL}{\operatorname{SL}}

\newcommand{\Tl}{\mathrm{T}_{\ell}}

\newcommand{\T}{\operatorname{{T}}}

\newcommand{\Fix}{\operatorname{Fix}}
\newcommand{\Sy}{\mathsf{S}}
\newcommand{\B}{\mathsf{B}}
\newcommand{\Pa}{\mathsf{P}}

\numberwithin{equation}{subsection}

\theoremstyle{plain}
\newtheorem{thm}[equation]{Theorem}

\newtheorem{lem}[equation]{Lemma}
\newtheorem{example}[equation]{Example}
\newtheorem{defn}[equation]{Definition}
\newtheorem{cor}[equation]{Corollary}
\newtheorem{prop}[equation]{Proposition}

\newtheorem{question}[equation]{Question}
\newtheorem*{qu}{Question}

\theoremstyle{remark}
\newtheorem{rmk}[equation]{Remark}

 \input xy
 \xyoption{all}

\def\sbtors{{}_{{\textup{tor}}}}

\begin{document}

\title[Torsion Subgroups of Abelian Surfaces]{Divisibility of torsion subgroups of abelian surfaces over number fields}
\author{John Cullinan and Jeffrey Yelton}
\address{Department of Mathematics, Bard College, Annandale-On-Hudson, NY 12401}
\email{cullinan@bard.edu}

\address{Department of Mathematics and Computer Science, Emory University,  Atlanta, GA 30322}
\email{jeffrey.yelton@emory.edu}
\begin{abstract}
Let $A$ be a 2-dimensional abelian variety defined over a number field $K$.  Fix a prime number $\ell$ and suppose $\#A(\Fp) \equiv 0 \pmod{\ell^2}$ for a set of primes $\p \subset \OK$ of density 1.  When $\ell=2$ Serre has shown that there does not necessarily exist a $K$-isogenous $A'$ such that $\#A'(K)\sbtors \equiv 0 \pmod{4}$.  We extend those results to all odd $\ell$ and classify the abelian varieties that fail this divisibility principle for torsion in terms of the image of the mod-$\ell^2$ representation.
\end{abstract}


\maketitle

\section{Introduction} \label{intro}

\subsection{Background} Let $A$ be an abelian variety defined over a number field $K$. If $\p$ is a prime of good reduction for $A$ and $m$ is a positive integer, then we say that $A$ \textbf{locally has a subgroup of order $m$} at $\p$ if $\#A(\Fp) \equiv 0 \pmod{m}$.  If $\p$ has absolute ramification index $e_\p < p-1$, then by \cite[Appendix]{katz}, the reduction-modulo-$\p$ map is injective on torsion:
$$
A(K)\sbtors \hookrightarrow A(\Fp).
$$
It follows that if $A(K)$ has a subgroup of order $m$ then it locally has a subgroup of order $m$ for a set of primes $\p$ of density 1.  On the other hand, if $A$ locally has a subgroup of order $m$ for a set of primes of density 1,  then it is not necessarily true that $A(K)$ has a global subgroup of order $m$.  For example, the elliptic curve with \textsf{LMFDB} label \textsf{11.a1} \cite{lmfdb_online} locally has a subgroup of order 5 for all $p \ne 11$, but has trivial Mordell-Weil group over $\Q$. Lang asked whether any abelian variety that locally has a subgroup of order $m$ for a set of primes of density 1 must be $K$-isogenous to one with a global subgroup of order $m$:

\begin{question}[Lang] \label{lq}
Let $A$ be an abelian variety defined over a number field $K$.  Suppose that $m \mid \#A_\p(\Fp)$ for a set of primes $\p$ of $K$ of density 1.  Does there exist a $K$-isogenous $A'$ such that $m \mid \#A'(K)\sbtors$?  
\end{question}

Note that if $m_1$ and $m_2$ are relatively prime integers and the answer to the above question is positive for a given abelian variety $A$ both when we take $m = m_1$ and when we take $m = m_2$, then the answer is positive for $A$ when we take $m = m_1 m_2$.  It therefore suffices to consider the question only in the case that $m$ is a prime power.

In \cite{katz}, Katz showed that the answer to Question \ref{lq} is affirmative when $A$ is an elliptic curve and, when $m$ is a prime number $\ell$, when $A$ is an abelian surface.  However, he showed by explicit construction that when $\dim A \geq 3$ and $m$ is odd the answer is negative (the degree $[K:\Q]$ of the field $K$ of definition of $A$ may be very large).  In \cite{2tors}, the first author considered the special case of $m=2$ and showed that the answer to Question \ref{lq} is affirmative when $\dim A = 3$ and negative when $\dim A \geq 4$. While we expect the answer when $\dim A=4$ and $m$ is a higher power of 2 to be negative as well, we do not pursue that question in this paper.  More generally, it may be of interest to determine whether a negative answer to Question \ref{lq} for fixed $\dim A$ and modulus $\ell^n$ implies a negative answer for modulus $\ell^m$ when $m >n$.

On the other hand, among all moduli $m = \ell^n$ and among all positive integers $\dim A$, there are two main cases where the answer to Question \ref{lq} is unknown: the case $\dim A = 3$ and $m=2^n$ ($n>1$), and the case $\dim A = 2$ and $m$ composite.   In all of these cases where there is a negative answer, it would be interesting to construct explicit examples of such abelian varieties where the degree $[K:\Q]$ of the field of definition of $A$ is minimized.  We refer to this as the \textbf{realization} problem and briefly address it at the end of this section.  

Returning to the open cases of Question \ref{lq} we focus exclusively on the situation where $m$ is composite and $\dim A = 2$ in this paper.  In an unpublished letter to Katz \cite{serretokatz}, Serre constructed a counterexample for the modulus 4 -- that is, he showed there exists an abelian surface $A$ that locally has a subgroup of order 4 for a set of primes of density 1 and no surface in the $K$-isogeny class of $A$ has a global subgroup of order 4.  More precisely, Serre constructed an open subgroup $H$ of the symplectic similitude group $\GSp_4(\Z_2)$ such that any abelian surface whose 2-adic image equals $H$ is one where Question \ref{lq} has a negative answer. However, Serre's construction does not immediately generalize to odd, composite moduli.  This is the starting point of our paper.

It is enough to answer Question \ref{lq} for prime-power moduli and in \cite[Introduction]{katz}, Katz shows that Question \ref{lq} is equivalent to Question \ref{kq} below (which in Katz's paper is \cite[Problem 1 (bis)]{katz}).  In order to state this equivalent version, we introduce some further notation.  Write $m = \ell^n$ for a prime number $\ell$ and a positive integer $n$.  Following Katz, we write $\Tl(A)$ for the $\ell$-adic Tate module of $A$ and $\rho_\ell:\Gal(\overline{K}/K) \to \Aut(\Tl(A))$ for the associated $\ell$-adic representation of $A$.  Now Question \ref{lq} is equivalent to the following.

\begin{question}[Katz] \label{kq}
Let $A$ be an abelian variety over a number $K$, $\ell$ a prime number, and $\rho_\ell: \Gal(\overline{K}/K) \to \Aut(\T_\ell(A))$ its $\ell$-adic representation.  Suppose we have 
$$
\det (\rho_\ell(\gamma)-1) \equiv 0 \pmod{\ell^n} \qquad \text{for all } \gamma \in \Gal(\overline{K}/K).
$$
Do there exist $\Gal(\overline{K}/K)$-stable lattices $\mathcal{L} \supset \mathcal{L}'$ in $\T_\ell(A)$ such that the quotient $\mathcal{L}/\mathcal{L}'$ has order $\ell^n$, and such that $\Gal(\overline{K}/K)$ acts trivially on $\mathcal{L}/\mathcal{L}'$?
\end{question}

Our main result is that the answer to Question \ref{kq} (and hence to Question \ref{lq}) is \emph{negative} for all moduli $\ell^2$ when $A$ is an abelian surface.  Our argument is purely group-theoretic; in the following section, we develop our general group-theoretic set-up.

\subsection{Restriction of the image} \label{sec1.2}  
Before stating our main theorem, we give a more detailed overview of our approach in this subsection.

Let $\mathcal{T}$ be a free $\Z_\ell$-module of rank $4$, which we equip with a nondegenerate alternating bilinear form 
$$
\langle \, , \, \rangle: \mathcal{T} \times \mathcal{T} \to \Zl.
$$
  Two groups of invertible operators on $\mathcal{T}$ that concern us are the group of symplectic similitudes, defined as  
$$\GSp(\mathcal{T}) := \{g \in \Aut(\mathcal{T}) \ | \ \langle g.v, g.w \rangle = m_g \langle v, w \rangle\},$$
 for some $m_g \in \Z_\ell^\times$ depending on $g$, and the symplectic group, defined as 
$$\Sp(\mathcal{T}) := \{g \in \Aut(\mathcal{T}) \ | \ \langle g.v, g.w \rangle = \langle v, w \rangle\} \subset \GSp(\mathcal{T}).$$  We may identify $\GSp(\mathcal{T})$ and $\Sp(\mathcal{T})$ with $\GSp_4(\Zl)$ and $\Sp_4(\Zl)$ respectively by fixing a symplectic basis of $\Tl(A)$ (see Definition \ref{dfn_symplectic} below).

We will be particularly interested in subgroups of $\Sp(\mathcal{T})$ satisfying a certain determinant condition, motivating us to define 
$$\Fix(\ell^n) := \lbrace \text{subgroups } G \subset \Sp(\mathcal{T}) \mid \det(g-1) \equiv 0 \pmod{\ell^n} \ \text{ for all $g \in G$} \rbrace.$$

\begin{rmk} \label{fix}
Our motivation for the above notation is that a subgroup $G \subset \Sp_4(\mathcal{T})$ lies in $\Fix(\ell^n)$ for some $n$ if and only if $G$ fixes an order-$\ell^n$ submodule of $\mathcal{T}$ under the induced mod-$\ell^n$ action.  Indeed, the latter condition is equivalent to saying that under this action, for each $g \in G$, the operator $g - 1 \in \Sp_4(\mathcal{T})$ kills a submodule of $(\mathcal{T} / \ell^n\mathcal{T})$ of order $\ell^n$.  We may assume that the image of $g - 1$ modulo $\ell^n$ is a diagonal matrix in $\End(\mathcal{T} / \ell^n \mathcal{T})$  after performing a series of invertible row and column operations on it.  It is then easy to see that in order for $g - 1$ to kill an order-$\ell^n$ submodule of $\mathcal{T} / \ell^n\mathcal{T}$, the product of its diagonal elements must be divisible by $\ell^n$, and so $\det(g - 1) \equiv 0$ (mod $\ell^n$), which is the defining criterion for membership in $\Fix(\ell^n)$.
\end{rmk}

We will show that the hypothesis that $G \in \Fix(\ell^2)$ is so strong that, with only a few exceptions, it forces the existence of pairs of $G$-stable lattices of relative index $\ell^2$ with trivial $G$-action on the quotient.  We will refer to these ``exceptions" using the following terminology.

\begin{defn} \label{dfn counterexample}
Given any free $\Z_\ell$-module $\mathcal{T}$ of rank $4$, we call any $G \in \Fix(\ell^2)$ for which there do not exist $G$-stable lattices $\mathcal{L}' \subset \mathcal{L} \subset \mathcal{T}$ with $[\mathcal{L}:\mathcal{L}'] = \ell^2$ and trivial $G$-action on the quotient $\mathcal{L}/\mathcal{L}'$ a \textbf{counterexample}.
\end{defn}

In the context of our investigation of Question \ref{kq}, the $\Z_\ell$-module $\mathcal{T}$ is the $\ell$-adic Tate module $T_\ell(A)$ of an abelian variety $A$ over a number field $K$ for some prime $\ell$.  After possibly replacing $A$ with an isogenous abelian variety, we assume that $A$ is principally polarized so that we may define the Weil pairing as a skew-symmetric form on $\Tl(A)$.  It is well known that the Weil pairing $\langle \, , \, \rangle$ on $T_\ell(A)$ is equivariant with respect to the action of the absolute Galois group $\Gal(\overline{K} / K)$ on $\Tl(A)$.  (See \cite[\S16]{milne} for more details on the Weil pairing.)  The image of the natural $\ell$-adic Galois representation $\rho_\ell:\Gal(\overline{K}/K) \to \Aut(\Tl(A))$ is therefore contained in $\GSp(T_\ell(A))$.  The following proposition explains our choice of the term ``counterexample" in Definition \ref{dfn counterexample} by showing that any counterexample $G$ in the sense of Definition \ref{dfn counterexample} gives a counterexample for Question \ref{kq} (and hence to Question \ref{lq}).

\begin{prop} \label{works_over_number_field}

Let $G \in \Fix(\ell^2)$ be a counterexample in the sense of Definition \ref{dfn counterexample}.  Then there is an abelian surface $A$ over a number field $K$ which provides a counterexample for Question \ref{kq} such that, taking $\mathcal{T}$ to be $T_\ell(A)$, we have $\im(\rho_\ell) \cap \Sp(\mathcal{T}) = G$.

\end{prop}

\begin{proof}

It is well known that for any prime $\ell$, there exists an abelian surface $A$ over $\Q$ such that the image of its natural $\ell$-adic Galois representation coincides with $\GSp(T_\ell(A))$.  In fact, we quote the stronger result of \cite[Theorem 1.3]{landesman_big_gal} that the image of the $\ell$-adic representation $\rho_\ell:\Gal(\overline{\Q}/\Q) \to \Aut(\Tl(A))$ of the hyperelliptic Jacobian defined by
$$
y^2 = x^6 + 7471225x^5 + 16548721x^4 + 6639451x^3 + 16857421x^2+ 20754195x+9508695
$$
is maximal for all $\ell$.

Now let $\Gamma'(\ell^2) \subset \GSp(\mathcal{T})$ be the subgroup consisting of all symplectic similitudes $g \in \GSp(\mathcal{T})$ such that $g \equiv 1$ (mod $\ell^2$), and let $G' \subset \GSp(T_\ell(A))$ be the subgroup generated by $G$ and $\Gamma'(\ell^2)$.  Let $K / \Q$ be the extension fixed by $\rho_\ell^{-1}(G')$; since $G'$ contains the finite-index subgroup $\Gamma'(\ell^2) \subset \GSp(T_\ell(A))$ and is therefore an open subgroup of $\GSp(T_\ell(A))$, the field extension $K / \Q$ is finite and so $K$ is a number field.  Now we consider $A$ as an abelian surface over $K$ and restrict $\rho_\ell$ to the absolute Galois group $\Gal(\overline{K} / K)$ of $K$; it is clear that we still have $\det(\rho_\ell(\gamma) - 1) \equiv 0$ (mod $\ell^2$) for all $\gamma \in \Gal(\overline{K} / K)$.  Meanwhile, since $G' \supset G$ and $G$ is a counterexample, there do not exist $\Gal(\overline{K}/K)$-stable lattices $\mathcal{L} \supset \mathcal{L}'$ in $\T_\ell(A)$ such that the quotient $\mathcal{L}/\mathcal{L}'$ has order $\ell^2$, and such that $\Gal(\overline{K}/K)$ acts trivially on $\mathcal{L}/\mathcal{L}'$, and we have therefore found a counterexample to Question \ref{kq} taking $n = 2$.
\end{proof}

By the discussion of \cite[pp.~482-483]{katz}, given an abelian surface $A$ arising from a counterexample $G \subset \Sp(\mathcal{T})$ via Proposition \ref{works_over_number_field}, after possibly replacing $A$ with an isogenous abelian surface, we may assume that its associated mod-$\ell$ Galois representation has a trivial $1$-dimensional subrepresentation.  In fact, we will see from Proposition \ref{prop_basic_shape} below that this property implies that the semisimplification of the mod-$\ell$ reduction $\overline{G}$ of a counterexample $G$ always has at least two $1$-dimensional factors.  We further distinguish between the following two cases: subgroups $G \subset \Sp(\mathcal{T})$ for which the semisimplification of the mod-$\ell$ reduction either 
\begin{enumerate}[(A)]
\item has four 1-dimensional factors, two of which are trivial, or
\item has an irreducible 2-dimensional factor and two trivial 1-dimensional factors.
\end{enumerate}
We remark that Serre's original counterexample for Question \ref{kq} in \cite{serretokatz} had four 1-dimensional factors and we will review this counterexample in \S\ref{sylowsection} below. 

Further, we will only mainly consider subgroups $G$ of $\Sp(\mathcal{T})$ where the kernel of reduction modulo $\ell$ is as large as possible (to be made precise below).  As we will see in \S\ref{lattices}, this assumption gives us a simple criterion for checking which lattices are $G$-stable and, therefore, whether there are any quotients of order $\ell^2$ with trivial $G$-action.

To summarize, in this paper for the purpose of finding and classifying counterexamples we will only consider subgroups $G \subset \Sp(\mathcal{T})$ such that 
\begin{itemize}
\item $\overline{G}$ fixes a subspace of $\mathcal{T}/\ell \mathcal{T}$ of dimension $1$; and
\item the kernel of the natural projection $G$ to its reduction modulo $\ell$ is as large as possible; for us, this will mean $G$ contains the full kernel $\Gamma(\ell)$ of the reduction-mod-$\ell$ map from $\Sp(\mathcal{T})$ to $\Sp(\mathcal{T} / \ell \mathcal{T})$ or contains a certain index-$\ell$ subgroup of $\Gamma(\ell)$.
\end{itemize}
What we will show using group theory is that even under these added hypotheses, counterexamples $G \in \Fix(\ell^2)$ do exist; more specifically, counterexamples satisfying property A exist for any $\ell$, and counterexamples satisfying property B exist when $\ell = 2$.  This in particular implies a negative answer to Question \ref{kq}.

\subsection{Statement of the Results}  With this motivation and background in place, we now state our main results.  For clarity of exposition we only give the \textit{maximal} counterexamples in this statement of the main theorem, while in the course of proving the results we give the minimal requirements that a counterexample must meet.

\begin{thm} \label{mainthm}
Let $\ell$ be a prime number; let $\mathcal{T}$ be a free $\Z_\ell$-module of rank $4$ equipped with a nondegenerate alternating pairing; and suppose $G \subset \Sp(\mathcal{T})$ satisfies
\begin{enumerate}
\item[(i)] $\det(g-1) \equiv 0\pmod{\ell^2}$ for all $g \in G$, and
\item[(ii)] there do not exist $G$-stable lattices $\mathcal{L}' \subset \mathcal{L} \subset \mathcal{T}$ of relative index $\ell^2$ with trivial $G$-action on the quotient $\mathcal{L}/\mathcal{L}'$ , and
\item[(iii)] $G$ is maximal among subgroups of $\Sp(\mathcal{T})$ satisfying (i) and (ii).
\end{enumerate}
Then one of the following holds.
\begin{enumerate}
\item[(a)] We have $\ell =2$, the image of $G$ modulo $2$ is isomorphic to $D_4 \times C_2$ or $S_3 \times C_2$, and we have $[\Gamma(2) : G \cap \Gamma(2)] = 2$.  In the former case, the semisimplification of the mod-$2$ representation consists of four 1-dimensional factors and in the latter it consists of two 1-dimensional factors and a 2-dimensional factor.
\item[(b)] We have $\ell = 2$, and $G$ is the full preimage in $\Sp(\mathcal{T})$ of a subgroup of $\Sp(\mathcal{T}/\ell\mathcal{T})$ isomorphic to $S_3$.  In this case the semisimplification of the mod-$2$ representation consists of two 1-dimensional factors and a 2-dimensional factor.
\item[(c)] We have $\ell \geq 3$, and $G$ is the full preimage in $\Sp(\mathcal{T})$ of a subgroup of $\Sp(\mathcal{T}/\ell\mathcal{T})$ of isomorphism type $\Z/\ell \rtimes (\Z/\ell)^\times$ or $\Z/\ell \times (\Z/\ell \rtimes (\Z/\ell)^\times)$.  In this case, the semisimplification of the mod-$\ell$ representation consists of four 1-dimensional factors.
\end{enumerate}
\end{thm}

In the course of proving Theorem \ref{mainthm} we give explicit representations, including the dimensions of the factors in the semisimplifications, of all of the groups that occur.  In Section \ref{S2.1} we give a more detailed exposition of how the group theoretic result of Theorem \ref{mainthm} implies the following corollary, answering Lang's original Question \ref{lq}.

\begin{cor}
For all square moduli, the answer to Question \ref{kq}, and hence to Question \ref{lq}, is negative. In addition, when the modulus is 4 there are counterexamples to Question \ref{lq} realized by absolutely simple abelian surfaces.
\end{cor}

Our paper proceeds as follows.  In \S\ref{background} we review the relevant background on symplectic groups and representation theory, and we give an explicit description of the stable lattice structure.  The classification of the maximal counterexamples then amounts to a group theory argument.  We break this argument up over the ensuing three sections.  In \S\ref{sylowsection} we classify the counterexamples for which $G$ is contained in a maximal pro-$\ell$ subgroup of $\Sp(\mathcal{T})$; we show that counterexamples of this type exist when $\ell=2$ and do not exist for $\ell \geq 3$.  This is where we will recall Serre's original counterexample.  Then in \S\ref{iwahorisection} we classify the counterexamples for which $G$ lies in the Iwahori subgroup of $\Sp(\mathcal{T})$; this is where the semisimplification of the mod-$\ell$ representation consists of four 1-dimensional factors.  Finally, \S\ref{parabolic_section} is where we consider the case where the semisimplification of the mod-$\ell$ representation contains an irreducible 2-dimensional factor.   There we show that no such counterexamples exist when $\ell \geq 3$ and classify the ones that do when $\ell=2$.  

\subsection{The Realization Problem} Our main theorem shows, from the point of view of group theory, that the answer to Question \ref{lq} is negative for abelian surfaces.  A follow-up question, which we call the \emph{realization problem}, is to construct abelian surfaces with these $\ell$-adic images.  By the Galois-theoretic argument in the proof of Proposition \ref{works_over_number_field}, we can construct such an abelian surface $A$ over a number field $K$ with $[K:\Q]$ large.  We therefore pose the following question.

\begin{qu}
Fix a prime number $\ell$.  What is the minimum degree $[K:\Q]$ such that there exists an abelian variety $A/K$ with $\ell$-adic Galois image equal to one of the groups in Theorem \ref{mainthm}?
\end{qu}

We do not pursue the realization problem in this paper, though we mention that the LMFDB's current database of over 68,000 genus 2 curves is a natural place to search for potential examples \cite{lmfdb}.  In particular, by sampling at a large number of primes, one can find examples of surfaces that have a local subgroup of (say) order $\ell^2$ but a global subgroup of order $\ell$.  For example, the isogeny class \href{http://www.lmfdb.org/Genus2Curve/Q/1270/a/}{1270.a} contains the hyperelliptic curve with label \href{http://www.lmfdb.org/Genus2Curve/Q/1270/a/325120/1}{1270.a.325120.1}.  By sampling a large number of primes, this Jacobian experimentally has a subgroup of order 4 (but a global subgroup of order 2), and the 2-torsion field has Galois group $S_3 \times S_2$. Furthermore, the semisimplification of the mod-2 representation contains an irreducible 2-dimensional factor and the Jacobian is absolutely simple over $\Q$.  

In order to rigorously determine whether or not a Jacobian defines a counterexample, one would need to determine both the mod-$\ell$ and mod-$\ell^2$ images of the $\ell$-adic representation.  A related question for future exploration would be to employ known estimates on the distribution of eigenvalues of Frobenius for genus-2 curves to determine how many primes one should sample to be ``reasonably" certain (in a precise sense) that an isogeny class consists entirely of Jacobians with a local subgroup of order $\ell^2$.

\subsection{Group Theory Notation} Given any prime $\ell$, suppose we have a free $\Z_\ell$-module $\mathcal{T}$ of rank $4$ equipped with a nondegerate alternating pairing $\langle \, , \, \rangle$.  Throughout this paper, for any integer $n \geq 1$, we denote the kernel of the reduction-modulo-$\ell^n$ map from $\Sp(\mathcal{T})$ to $\Sp(\mathcal{T} / \ell^n \mathcal{T})$ by
$$\Gamma(\ell^n) := \{g \in \Sp(\mathcal{T}) \ | \ g \equiv 1 \ \ (\mathrm{mod} \ \ell^n)\}.$$

The form $\langle \, , \, \rangle$ induces a nondegenerate alternating pairing on $\mathcal{T} / \ell^n \mathcal{T}$ via reduction modulo $\ell^n$ for each $n \geq 1$.  
We have (surjective)  homomorphisms $\pi_{\ell^n} : \Sp(\mathcal{T}) \to \Sp(\mathcal{T}/\ell^n\mathcal{T})$ for all $n \geq 1$, as well as (surjective) homomorphisms $\pi_{\ell^n \to \ell^m} : \Sp(\mathcal{T}/\ell^n\mathcal{T}) \to \Sp(\mathcal{T}/\ell^m\mathcal{T})$ for any integers $n > m \geq 1$.  

We use both the notations $C_n$ and $\Z/n$ for a cyclic group of order $n$ depending on whether we are emphasizing a multiplicative or additive group structure, respectively.  The notation $S_n$ and $A_n$ refer to the symmetric and alternating groups on $n$ letters, respectively.  We use $D_n$ to denote the dihedral group of order $2n$. \\

\subsection{Acknowledgments.}

The authors would like to extend their gratitude to J.-P. Serre for providing the first author with his (unpublished) letter to Katz (\cite{serretokatz}) which in large part inspired this whole project, and for graciously giving permission to discuss it in this work.  We would also like to thank Drew Sutherland for helpful comments and to thank the anonymous referees for many helpful comments and clarifications.

\section{Symplectic Groups} \label{background}

The basic definitions and background on symplectic groups are widely available in the literature: see \cite{book} for a general development or \cite{landesman} for concise definitions, especially for symplectic groups over general commutative rings.  We will be content with a very brief overview here.  Let $\ell$ be a prime number and $\mathcal{T}$ be a free rank-$4$ $\Zl$-module equipped with a nondegenerate alternating bilinear form $\langle \, , \, \rangle$ as in \S\ref{sec1.2}.  
Many of our explicit group-theoretic arguments are performed in coordinates and so we make the following formal definition.

\begin{defn} \label{dfn_symplectic}

A \textbf{symplectic} basis of $\mathcal{T}$ is an ordered basis $\{e_1, e_2, e_3, e_4\}$ that satisfies $\langle e_1, e_4 \rangle = \langle e_2, e_3 \rangle = 1$ and $\langle e_1, e_2 \rangle = \langle e_1, e_3 \rangle = \langle e_4, e_2 \rangle = \langle e_4, e_3 \rangle = 0$.

\end{defn}

To ease notation, write 
$$
\overline{G} := \pi_\ell(G)
$$
for any subgroup $G \subset \Sp(\mathcal{T})$; this reduction $\overline{G}$ acts naturally on the $4$-dimensional symplectic $\Fl$-vector space $\mathcal{T}/\ell\mathcal{T}$.  Note that $\Gamma(\ell) \lhd \Sp(\mathcal{T})$ is the kernel of $\pi_\ell$.  Some of the counterexamples $G$ of Theorem \ref{mainthm} are maximal in the sense that $\Gamma(\ell) \subset G$; such a subgroup $G$ can be described simply as the inverse image under $\pi_{\ell}$ of $\overline{G}$.

\subsection{The mod-$\ell$ Representation} \label{S2.1}

We now determine the ``shape" of $\overline{G}$, where $G \in \Fix(\ell^2)$ is a counterexample satisfying the assumptions given in \S\ref{sec1.2}.

\begin{prop} \label{prop_basic_shape}

Let $G \in \Fix(\ell^2)$ be a counterexample whose reduction $\overline{G}$ acts trivially on a $1$-dimensional subspace of $\mathcal{T} / \ell \mathcal{T}$.  Then, with respect to the reduction of a suitable symplectic basis of $\mathcal{T}$, the elements of $\overline{G}$ are all matrices of the form 
\begin{align} \label{basic_shape}
\left(
\begin{smallmatrix}
1 & 0 & 0 & 0 \\
* & * & * & 0 \\
* & * & * & 0 \\
* & * & * & 1 \\
\end{smallmatrix}
\right).
\end{align}

\end{prop}

\begin{proof}

The statement amounts to asserting that there exists a $\overline{G}$-invariant subspace $\overline{W} \subset \mathcal{T} / \ell \mathcal{T}$ of dimension $3$ such that $\overline{G}$ acts trivially on the quotient $(\mathcal{T} / \ell \mathcal{T}) / \overline{W}$.  We now construct a symplectic basis $\lbrace e_1,e_2,e_3,e_4 \rbrace$ of $\mathcal{T}$ such that the reduction $\overline{e}_4 \in \mathcal{T}/\ell\mathcal{T}$ of $e_4$ generates the $1$-dimensional subspace which is fixed by $\overline{G}$.  Since the symplectic pairing is nondegenerate, there is an element $e_1 \in \mathcal{T}$ such that $\langle e_1, e_4 \rangle = 1$.  Let $V$ be the span of $\{e_1, e_4\}$.  Since the sum of the dimensions of any subspace of $\mathcal{T}$ and its orthogonal complement must equal the dimension of $\mathcal{T}$, we get that the orthogonal complement $V^{\bot}$ of $V$ has dimension $2$.  Now clearly there is a nondegenerate symplectic pairing on $V / (V \cap V^{\bot})$ induced by $\langle \, , \, \rangle$, so $V / (V \cap V^{\bot})$ must have even dimension; since $\langle \,, \, \rangle$ is nontrivial on $V$, we must have $V^{\bot} \neq V$ and so $V \cap V^{\bot} = \{0\}$.  Then since $V \oplus V^{\bot} = \mathcal{T}$, the pairing $\langle \,, \, \rangle$ cannot be trivial on $V^{\bot}$, and any ordered basis $\{e_2, e_3\}$ satisfies $\langle e_2, e_3 \rangle = 1$ after appropriate scaling of one of the elements.  Thus we have an ordered set $\{e_1, e_2, e_3, e_4\}$ which is a symplectic basis of $\mathcal{T}$.

Now let $W \subset \mathcal{T}$ be the span of $\{e_2, e_3, e_4\}$.  It is clear that $W$ is the orthogonal complement of the subspace spanned by $e_4$.  It now follows from the $G$-invariance of the symplectic pairing that since $\overline{G}$ fixes the element $\overline{e}_4$, the reduction $\overline{W}$ of $W$ is invariant under the action of $\overline{G}$.  Note further that for any $g \in G$, we have $\langle g(e_1), g(e_4) \rangle = \langle e_1, e_4 \rangle = 1$.  It follows that since $g(e_4) \equiv e_4$ (mod $\ell$), we have $\langle g(e_1), e_4 \rangle \equiv 1$ (mod $\ell$) and therefore $g(e_1) \equiv e_1 + w$ (mod $\ell$) for some $w \in W$.  Since $\overline{W}$ and the image of $\overline{e}_1$ modulo $\ell$ generate $\mathcal{T} / \ell \mathcal{T}$, we get the desired statement.
\end{proof}

Given any counterexample $G \in \Fix(\ell^2)$ satisfying the hypothesis of the above proposition, we fix a symplectic basis $\{e_1, e_2, e_3, e_4\}$ of $\mathcal{T}$ such that, with respect to its reduction modulo $\ell$ (which we denote by $\{\overline{e}_1, \overline{e}_2, \overline{e}_3, \overline{e}_4\}$ and which is a basis of $\mathcal{T} / \ell \mathcal{T}$), the elements of the group $\overline{G}$ are matrices of the form given in (\ref{basic_shape}).  (Note that since $G \subset \Sp(\mathcal{T})$, there are additional constraints placed on the off-diagonal entries by $\langle \, , \,\rangle$.)

In order to classify possible counterexamples $G \in \Fix(\ell^2)$, we will now introduce three subgroups $\Sy_\ell \subset \B_\ell \subset \Pa_\ell \subset \Sp(\mathcal{T})$.  We first define $\Pa_\ell$ to be the subgroup of matrices whose reductions modulo $\ell$ are of the lower block-triangular form given in (\ref{basic_shape}).  According to Proposition \ref{prop_basic_shape} above, any counterexample $G \in \Fix(\ell^2)$ is conjugate, in $\Sp(\mathcal{T})$, to a subgroup of $\Pa_\ell$.

For any element of $\Pa_\ell$, we highlight the off-diagonal lower entries in the matrix form of its reduction modulo $\ell$ as 
\begin{align} \label{parabolic_block2}
\left( \begin{smallmatrix} 1&0&0&0 \\ \alpha & * & * &0 \\ \beta & \gamma & * & 0 \\ \delta & \beta' & \alpha' & 1 \end{smallmatrix} \right),
\end{align}
 and we let $\alpha, \alpha', \beta, \beta', \gamma, \delta : \Pa_\ell \to \Z / \ell$ be the maps given by taking an element $g \in \Pa_\ell$ to the corresponding entries of $\pi_{\ell}(g)$ shown in (\ref{parabolic_block2}).  (These maps are not homomorphisms in general.)  Noting from (\ref{parabolic_block2}) that the (1,4) entry $g_{1,4}$ of any element $g \in \Pa_\ell$ is divisible by $\ell$, we also define a map 
$$
f : \Pa_\ell \to \Z / \ell, \ \ g \mapsto \pi_{\ell}(g_{1,4} / \ell).
$$  
It is easy to verify that this map $f$ is a homomorphism.

The \textbf{Iwahori subgroup} $\B_\ell \subset \Sp(\mathcal{T})$ is the subgroup consisting of all matrices whose reduction modulo $\ell$ is lower-triangular.  We write $\Sy_\ell \subset \B_\ell$ for the (maximal pro-$\ell$-) \textbf{Sylow subgroup} of $\B_\ell$; it is the subgroup of $\B_{\ell}$ consisting of those triangular matrices whose reduction modulo $\ell$ has all $1$'s along the diagonal.  Given a group $G \in \Fix(\ell)$ we write $G \subset \B_{\ell}$ (resp. $G \subset \Sy_{\ell}$) if there is a basis $\{e_1, e_2, e_3, e_4\}$ with the properties given above with respect to which $G$ is contained in $\B_{\ell}$ (resp. $\Sy_{\ell}$).  With respect to this chosen basis, a calculation with $\langle \, , \, \rangle$ shows that if we let $\epsilon : \B_\ell \to (\Z / \ell)^{\times}$ be the map taking an element $g \in \B_\ell$ to the $(2,2)$ entry of $\pi_\ell(g)$, the $(3,3)$ entry of $\pi_{\ell}(g)$ is given by $1 / \epsilon(g)$, and that we have the formulas below:

\begin{align}\label{prime}
\alpha'(g) &= -\alpha(g) / \epsilon(g) \\ 
\beta'(g) &= \beta(g)\epsilon(g) - \alpha(g)\gamma(g). \nonumber
\end{align}
  We note that the map $\epsilon : \B_\ell \to (\Z / \ell)^{\times}$ is a homomorphism whose kernel coincides with $\Sy_{\ell}$.  For any $g \in \B_\ell$, the determinant $\det(g-1)$ satisfies
\begin{align} \label{iwahori_det}
\det(g-1) \equiv  \left(\gamma(g)\alpha(g)^2 + 
\frac{\beta(g) \alpha(g)(1-\epsilon(g)^2)}{\epsilon(g)}  +
 \frac{\delta(g)(1-\epsilon(g))^2}{\epsilon(g)}
\right)f(g) \ell\pmod{\ell^2}.
\end{align}
  
We further note that when we restrict to $\Sy_{\ell}$, the maps $\alpha, \beta : \Sy_\ell \to \Z / \ell$ become homomorphisms, and we get the simplified formulas from (\ref{prime}) below:
\begin{align} \label{prime_simplified}
\alpha'(g) &= -\alpha(g) \\ 
\beta'(g) &= \beta(g) - \alpha(g)\gamma(g). \nonumber
\end{align}
In this special case, the determinant formula (\ref{iwahori_det}) simplifies to 
\begin{align} \label{det_formula}
\det(g-1) \equiv \alpha(g)^2\gamma(g)f(g) \ell \pmod{\ell^2}.
\end{align}

\subsection{Lattices} \label{lattices}  As we see from the conditions given in Question \ref{kq}, in order to determine whether a given group $G \in \Sp(\mathcal{T})$ is a counterexample, it is crucial to understand the stable lattice structure of $\mathcal{T}$ under the action of $G$.  The following proposition will allow us to essentially work with the mod-$\ell$ representation and search for pairs of stable subspaces of $\mathcal{T} / \ell \mathcal{T}$ whose quotients admit trivial $G$-action rather than search through all sublattices of $\mathcal{T}$.  

\begin{prop} \label{containslT}

Assume that $G \in \Fix(\ell^2)$ contains $\Gamma(\ell) \cap \ker(f)$ and that $f$ is not trivial on $G$.  Let $\mathcal{L} \subset \mathcal{T}$ be a $G$-stable lattice.  Then we have $\mathcal{L} \subset \ell \mathcal{T}$ or $\mathcal{L} \supset \ell \mathcal{T}$.

\end{prop}

In order to prove the above proposition, we first need a lemma.

\begin{lem} \label{containslTlemma}

Choose a vector $v \in \mathcal{T} \smallsetminus \ell \mathcal{T}$.

a) For any integer $n \geq 1$, the orbit of $v$ under the action of the subgroup $\Gamma(\ell^n) \subset \Sp(\mathcal{T})$ coincides with the coset $v + \ell^n \mathcal{T}$.

b) Let $W \subset \mathcal{T}$ be the submodule consisting of vectors $w$ such that $\langle w, e_4 \rangle \in \ell \Z_{\ell}$.  The orbit of $v$ under the action of the subgroup $\Gamma(\ell) \cap \ker(f) \subset \Sp(\mathcal{T})$ coincides with the coset $v + W$ (resp. $v + \ell \mathcal{T}$) if $v$ is (resp. is not) a (unit) scalar multiple of $e_4 \in \mathcal{T}$.

\end{lem}

\begin{proof}

In order to prove part (a), we first show that for any integer $m \geq 1$, the orbit of any vector $v \in \mathcal{T} / \ell^{m+1} \mathcal{T}$ under the action of the subgroup $\Gamma(\ell^m) / \Gamma(\ell^{m+1}) \subset \Sp(\mathcal{T} / \ell^{m+1} \mathcal{T})$ consists of all vectors $v' \in \mathcal{T} / \ell^{m+1} \mathcal{T}$ equivalent to $v$ modulo $\ell^m$.  (Below we abuse notation slightly and write $\langle \, , \, \rangle$ for the pairing on $\mathcal{T} / \ell^{m+1} \mathcal{T}$ induced by reducing the symplectic pairing on $\mathcal{T}$ modulo $\ell^{m+1}$; it takes values in $\Z / \ell^{m+1}$.)  In order to do this, we define, for any vector $u \in \mathcal{T} / \ell^{m+1} \mathcal{T}$, the (unipotent) operator $T_u \in \Gamma(\ell^m) / \Gamma(\ell^{m+1})$ given by $w \mapsto w + \ell^m \langle w, u \rangle u$ for $w \in \mathcal{T} / \ell^{m+1} \mathcal{T}$.  Now choose $a \in (\mathcal{T} \smallsetminus \ell \mathcal{T}) / \ell^{m+1} \mathcal{T}$; we proceed to show that $v + \ell^m a$ lies in the orbit of $v$ under $\Gamma(\ell^m) / \Gamma(\ell^{m+1})$.  First assume that $\langle v, a \rangle \not\equiv 0 \ (\mathrm{mod} \ \ell)$.  In this case, we clearly have $T_u^{\langle v, a \rangle^{-1}}(v) = v + \ell^m a$ and we are done.  Now assume that $\langle v, a \rangle \equiv 0 \ (\mathrm{mod} \ \ell)$.  Then one sees from a simple dimension-counting argument that there is a vector $b \in (\mathcal{T} \smallsetminus \ell \mathcal{T}) / \ell^{m+1} \mathcal{T}$ satisfying $\langle v, b \rangle \equiv 1 \ (\mathrm{mod} \ \ell)$ and $\langle b, a \rangle \not\equiv 0 \ (\mathrm{mod} \ \ell)$.  Now we compute $(T_{a+b}^{\langle b, a \rangle^{-1}} T_{b}^{-1})(v) = v + \ell^m a$, and we are done proving our claim about the orbit of $v$ under $\Gamma(\ell^m)$ modulo $\ell^{m + 1}$.

Now we fix $n \geq 1$ and claim that for any $n' > n \geq 1$ the orbit of any $v \in \mathcal{T} / \ell^{n'} \mathcal{T}$ under the action of the subgroup $\Gamma(\ell^n) / \Gamma(\ell^{n'}) \subset \Sp(\mathcal{T} / \ell^{n'} \mathcal{T})$ consists of all vectors $v' \in \mathcal{T} / \ell^{n'} \mathcal{T}$ equivalent to $v$ modulo $\ell^{n'}$.  We prove this claim by performing induction on $n'$, noting that we get the $n' = n + 1$ case by applying the statement we proved in the last paragraph with $m = n$.  Now if our claim is true for a particular $n' > n$, we see that it holds for $n' + 1$ as well: indeed, for any vectors $v, w \in \mathcal{T} / \ell^{n' + 1} \mathcal{T}$ with $v \equiv w$ (mod $\ell^{n'}$), the inductive assumption provides an operator $g \in \Gamma(\ell^{n}) / \Gamma(\ell^{n' + 1})$ which takes $v$ to some $w'$ which is equivalent to $w$ modulo $\ell^{n'}$; then by applying the statement we proved in the last paragraph with $m = n'$ we get an operator $h \in \Gamma(\ell^{n'}) / \Gamma(\ell^{n' + 1}) \subset \Gamma(\ell^n) / \Gamma(\ell^{n' + 1})$ which takes $w'$ to $w$.  Therefore, $w$ is in the orbit of $v$ under the action of $\Gamma(\ell^n) / \Gamma(\ell^{n' + 1})$, and our claim is proved.  Now by moving to the inverse limit of the groups $\Gamma(\ell^n) / \Gamma(\ell^{n'})$ and the modules $\mathcal{T} / \ell^{n'}\mathcal{T}$, we get the statement of part (a).

Now by applying part (a) for $n = 2$, it is clear that in order to prove part (b), it suffices only to consider the orbit of a vector $v \in (\mathcal{T} \smallsetminus \ell \mathcal{T}) / \ell^2 \mathcal{T}$ under the action of the subgroup $(\Gamma(\ell) \cap \ker(f)) / \Gamma(\ell^2)$.  First assume that $v$ is (the image modulo $\ell^2$ of) a scalar multiple of $e_4$.  Then it follows immediately from the definition of $f$ that an operator $T \in \Gamma(\ell) / \Gamma(\ell^2)$ lies in $\ker(f) / \Gamma(\ell^2)$ if and only if $\langle v, T(v) \rangle = \langle v, T(v) - v \rangle = 0$, or equivalently, if $T(v) = v + \ell w$ for some $w \in W$, whence the first statement of (b).  Now assume that $v$ is not (the image modulo $\ell^2$ of) a scalar multiple of $e_4$.  Then there exists a vector $b \in (\mathcal{T} \smallsetminus \ell \mathcal{T}) / \ell^{n+1} \mathcal{T}$ satisfying $\langle e_4, b \rangle \not\equiv 0 \ (\mathrm{mod} \ \ell)$ and $\langle v, b \rangle \equiv 0 \ (\mathrm{mod} \ \ell)$; the first condition implies that $T_b \notin (\Gamma(\ell) \cap \ker(f)) / \Gamma(\ell^2)$, while the second condition implies that $T_b$ fixes $v$.  We know from part (a) that there exists an operator $T \in \Gamma(\ell) / \Gamma(\ell^2)$ such that $T(v) = v + \ell a$ for any given vector $a \in (\mathcal{T} \smallsetminus \ell \mathcal{T}) / \ell^2 \mathcal{T}$.  Since $\Gamma(\ell) / \Gamma(\ell^2)$ is cyclically generated over $(\Gamma(\ell) \cap \ker(f)) / \Gamma(\ell^2)$, the operator $T T_b^m$ lies in $(\Gamma(\ell) \cap \ker(f)) / \Gamma(\ell^2)$; the second claim of part (b) follows from the fact that $(T  T_b^m)(v) = T(v) = v + \ell a$.
\end{proof}

\begin{proof}[Proof of Proposition \ref{containslT}]
We assume that $\mathcal{L} \not\subset \ell \mathcal{T}$ and proceed to show that $\mathcal{L}$ must contain $\ell \mathcal{T}$.  Choose a vector $v \in \mathcal{L} \smallsetminus \ell \mathcal{T}$.  Since we have $\Gamma(\ell) \cap \ker(f) \subset G$ by hypothesis, we may apply Lemma \ref{containslTlemma}(b) to get that $v + W \subset \mathcal{L}$ (resp. $v + \ell \mathcal{T} \subset \mathcal{L}$) if $v$ is (resp. is not) a scalar multiple of $e_4$, where $W \subset \mathcal{T}$ is as defined in the statement of the lemma.  Since $\mathcal{L}$ is closed under addition, we immediately get $\mathcal{L} \supset \ell W$ (resp. $\mathcal{L} \supset \ell \mathcal{T}$, in which case we are done).  Now suppose that we are in the former case; we assume without loss of generality that $v = e_4$.  Since $f$ is nontrivial on $G$, we may choose some element $y \in G \smallsetminus \ker(f)$.  As $G$ has the block-upper-triangular structure described above, we have that $y(e_4) - e_4 \in \ell \mathcal{T}$; the fact that $f(y) \neq 0$ then implies that we have $y(e_4) - e_4 \in \ell \mathcal{T} \smallsetminus W$.  Since $\mathcal{T}$ is generated over $W$ by any element in $\mathcal{T} \smallsetminus W$, we get the desired inclusion $\mathcal{L} \supset \ell \mathcal{T}$.
\end{proof}

Let $G$ be a group in $\Fix(\ell^2)$, and for $0 \leq i \leq 3$, define the sublattice 
\begin{align} \label{latticedef}
L_i := {\rm span}_{\Zl} \lbrace \ell e_1,\dots,\ell e_i, e_{i+1}, \dots, e_4  \rbrace \subset \mathcal{T}. 
\end{align}
The following proposition will be useful below for determining whether a given $G \in \Fix(\ell^2)$ is a counterexample or not.

\begin{prop} \label{latticeprop}

a) The sublattices $L_1$ and $L_3$ are always $G$-stable.  Moreover, the quotient $L_3 / \ell L_1$ is fixed under the induced $G$-action if and only if the homomorphism $f : G \to \Z / \ell$ is trivial.  Thus, if $G$ is a counterexample, then $f$ is nontrivial on $G$.

b) If we have $G \subset \B_{\ell}$, then $L_2$ is also a $G$-invariant sublattice.  Suppose further that we have $G \subset \Sy_{\ell}$.  Then the quotient $L_1 / L_3$ (resp. the quotients $L_0 / L_2$ and $L_2 / \ell L_0$) is fixed under the induced $G$-action if and only if the homomorphism $\gamma : G \to \Z / \ell$ (resp. $\alpha : G \to \Z / \ell$) is trivial.  Thus, if $G \subset \Sy_{\ell}$ is a counterexample, then both $\alpha$ and $\gamma$ are nontrivial on $G$.

c) In order to verify that $G$ is a counterexample, it suffices to verify that for any $G$-invariant sublattices $\mathcal{L}' \subset \mathcal{L} \subset \mathcal{T}$ both containing $\ell \mathcal{T}$ and with quotient of order $\ell^2$, the induced action of $G$ on the quotients $\mathcal{L} / \mathcal{L}'$ and $\mathcal{L}' / \ell \mathcal{L}$ is not trivial.

\end{prop}

\begin{proof}

The statements of parts (a) and (b) can be verified straightforwardly from the discussion and definitions in \S\ref{S2.1}.

We proceed to prove part (c).  Choose any group $G \in \Fix(\ell^2)$ such that there exist $G$-invariant lattices $\mathcal{M}' \subset \mathcal{M} \subset \mathcal{T}$ whose quotient $\mathcal{M} / \mathcal{M}'$ has order $\ell^2$ and is fixed under the induced action by $G$.  If we have $\mathcal{M} \subset \ell \mathcal{T}$, then we may replace $\mathcal{M}$ and $\mathcal{M}'$ with $\frac{1}{\ell}\mathcal{M}$ and $\frac{1}{\ell}\mathcal{M}'$ respectively without changing the induced action of $G$ on their quotient.  We therefore assume that $\mathcal{M} \not\subset \ell \mathcal{T}$, which by part (a) combined with Proposition \ref{containslT} implies that $\mathcal{M} \supsetneq \ell \mathcal{T}$.  If we also have $\mathcal{M}' \not\subset \ell \mathcal{T}$, then we similarly get $\mathcal{M}' \supsetneq \ell \mathcal{T}$.  In this case, we let $\mathcal{L} = \mathcal{M}$ and $\mathcal{L}' = \mathcal{M}'$, and we are done.

Now assume that $\mathcal{M}' \subset \ell \mathcal{T}$.  Then it follows from considering the order of the quotient $\mathcal{M} / \mathcal{M}'$ that we certainly have $\mathcal{M}' \not\subset \ell^2 \mathcal{T}$, implying that $\frac{1}{\ell} \mathcal{M}' \not\subset \ell \mathcal{T}$.  Now by part (a) combined with Proposition \ref{containslT}, this implies that $\frac{1}{\ell} \mathcal{M}' \supsetneq \ell \mathcal{T}$ and so we have the inclusions $\ell^2 \mathcal{T} \subsetneq \mathcal{M}' \subset \ell \mathcal{T}$.  Suppose that $\mathcal{M} / \mathcal{M}' \cong \Z / \ell^2$, so that there exists an element $v \in \mathcal{M} \smallsetminus \ell \mathcal{T}$ whose image modulo $\mathcal{M}'$ generates $\mathcal{M} / \mathcal{M}'$.  If $v \equiv e_4 \ (\mathrm{mod} \ \ell)$, then one verifies directly from the definition of $f$ that since $v$ is fixed modulo $\mathcal{M}'$ by $G$, the homomorphism $f$ must be trivial on $G$; then by part (a) we may take $\mathcal{L} = L_3$ and $\mathcal{L}' = \ell L_1$ and we are done.  If, on the other hand, we have $v \not\equiv e_4 \ (\mathrm{mod} \ \ell)$, we may take $\mathcal{L}$ to be generated by $\ell \mathcal{T}$ and the elements $v$ and  $e_4$; it is immediate to check that $G$ acts trivially on $\mathcal{L} / \mathcal{L}'$ and again we are done.  Now finally, suppose that $\mathcal{M} / \mathcal{M}' \cong \Z / \ell \oplus \Z / \ell$.  In this case, we clearly have $\mathcal{M} \subset \frac{1}{\ell}\mathcal{M}' \subset \mathcal{T}$ and can therefore take $\mathcal{L} = \mathcal{M}'$ and $\mathcal{L}' = \frac{1}{\ell}\mathcal{M}$, finishing the proof of part (c).
\end{proof}

The following lemma will be useful in both \S\ref{sylowsection} and \S\ref{iwahorisection}.

\begin{lem} \label{jeffsylowlemma}
Let $G$ be any subgroup of $\B_\ell$ such that the homomorphisms $\alpha$ and $\gamma$ are nontrivial on $G$.  Then the only proper $G$-stable sublattices of $\mathcal{T}$ which properly contain $\ell\mathcal{T}$ are $L_0$, $L_1$, $L_2$, $L_3$, where the sublattices $L_i$ are as in (\ref{latticedef}).  Moreover, if $G \in \Fix(\ell^2)$ and the homomorphism $f$ is also trivial on $G$, then $G$ is a counterexample.
\end{lem}

\begin{proof}

For $i = 0, 1, 2, 3$, we write $\overline{e}_i \in \mathcal{T} / \ell\mathcal{T}$ for the reduction modulo $\ell$ of $e_i \in \mathcal{T}$ and $\overline{L}_i \subset \mathcal{T} / \ell\mathcal{T}$ for the $\overline{G}$-invariant subspace given by $L_i / \ell\mathcal{T}$, so that $\overline{L}_i = \langle \overline{e}_{i + 1}, ... , \overline{e}_4 \rangle$.  The first statement of the lemma is equivalent to saying that the only nontrivial $\overline{G}$-invariant subspaces of $\mathcal{T} / \ell\mathcal{T}$ are the $\overline{L}_i$'s.  We prove this by showing that for $i = 0, 1, 2$, given a vector $v \in \overline{L}_i \smallsetminus \overline{L}_{i + 1}$, the minimal $\overline{G}$-invariant subspace containing $v$ is $\overline{L}_i$.  We start by choosing $v \in \overline{L}_2 \smallsetminus \overline{L}_3$; on choosing some $g \in G \smallsetminus \ker(\alpha)$, we get that $\epsilon(g)^{-1}v - \pi_{\ell}(g).v \in \langle \overline{e}_4 \rangle = \overline{L}_3$.  Any $\overline{G}$-invariant subspace containing $v$ therefore contains the subspace generated by $\overline{L}_3$ and $v$, which coincides with $\overline{L}_2$.  We have thus proved that the minimal $\overline{G}$-invariant subspace containing $v$ is $\overline{L}_2$.  We now show that for $v \in \overline{L}_1 \smallsetminus \overline{L}_2$, the minimal $\overline{G}$-invariant subspace containing $v$ is $\overline{L}_1$, using a similar argument where this time we choose some $g \in G \smallsetminus \ker(\gamma)$ and take $\epsilon(g)v - \pi_{\ell}(g).v$.  Finally we show that for $v \in \overline{L}_0 \smallsetminus \overline{L}_1$, the minimal $\overline{G}$-invariant subspace containing $v$ coincides with $\overline{L}_0$ in the same way, this time choosing some $g \in G \smallsetminus \ker(\alpha)$ and taking $v - \pi_{\ell}(g).v$.

The second statement now follows easily by applying all three parts of Proposition \ref{latticeprop}.
\end{proof}

\section{The Sylow Subgroup $\Sy_\ell$} \label{sylowsection}

In this section we consider subgroups $G \subset \Sy_\ell$ that belong to $\Fix(\ell^2)$.  Throughout this section, by fixing an appropriate symplectic basis of our free rank-$4$ $\Z_\ell$-module $\mathcal{T}$, we identify $\Sy_\ell$ with the subgroup of $\Sp_4(\Z_\ell)$ whose reduction modulo $\ell$ consists of lower-triangular matrices with only $1$'s on the diagonal. 

We have two main results: that there are no counterexamples when $\ell \geq 3$, and that there exist counterexamples when $\ell=2$.

\begin{thm} \label{jointsylowthm}

Suppose $G \subset \Sy_\ell$ with $G \in \Fix(\ell^2)$.

a) If $\ell \geq 3$, then one of $\alpha$, $\gamma$, or $f$ is trivial on $G$, \emph{i.e.}~there are no counterexamples when $\ell \geq 3$ and $G \subset \Sy_\ell$.   \label{geq3sylow}

b) Let $G \subset \Sy_2$ be a counterexample.  Then we have either $\overline{G} = \overline{\Sy}_2 \cong C_2 \times D_4$ or $\overline{G} \cong D_4$.  In either case, for any $H \subset \Gamma(2) \cap \ker(f)$, the subgroup of $\Sp_4(\mathcal{T})$ generated by $G$ and $H$ is also a counterexample.  In particular, if $G$ is a maximal counterexample, then we have $\overline{G} = D_4 \times C_2$ and $G \cap \Gamma(2) = \Gamma(2) \cap \ker(f)$.

Moreover, there do exist counterexamples satisfying $\overline{G} = \overline{\Sy}_2$ and counterexamples satisfying $\overline{G} = D$ for any subgroup $D \subset \overline{\Sy}_2$ isomorphic to $D_4$. \label{jeffsylowprop}

\end{thm}

Because of the difference in techniques of the two cases of Theorem \ref{jointsylowthm}, we separate our argument into multiple subsections.  We begin with a short description of the structure of $\overline{\Sy}_\ell \subset \Sp_4(\Z/\ell)$ that we will use extensively throughout this section.

\subsection{The structure of $\Sy_\ell$}  \label{structure_of_sy} Let $\Sy_\ell$ be the $\ell$-Sylow subgroup of $\B_\ell$  and $\overline{\Sy}_\ell$ the $\ell$-Sylow subgroup of $\overline{\B}_\ell$.  We define the following four elements of $\overline{\Sy}_\ell$ that we will use extensively in the rest of the paper:
$$
x_1= \left( \begin{smallmatrix} 1&0&0&0 \\ 1&1 &0&0 \\ 0&0&1&0 \\ 0&0&-1&1  \end{smallmatrix} \right)
\qquad
x_2 = \left( \begin{smallmatrix} 1&0&0&0 \\ 0&1 &0&0 \\ 0&1&1&0 \\ 0&0&0&1  \end{smallmatrix} \right)
\qquad
x_3 = \left( \begin{smallmatrix} 1&0&0&0 \\ 0&1 &0&0 \\ 1&0&1&0 \\ 0&1&0&1  \end{smallmatrix} \right)
\qquad
x_4 = \left( \begin{smallmatrix} 1&0&0&0 \\ 0&1 &0&0 \\ 0&0&1&0 \\ 1&0&0&1  \end{smallmatrix} \right).
$$
The group $\overline{\Sy}_\ell$ is nonabelian of order $\ell^4$ and the following facts are easily verified to hold for all $\ell$. It is straightforward to compute directly that the order-$\ell$ elements $x_2,x_3,x_4$ commute and so $\langle x_2,x_3,x_4 \rangle$ defines an elementary abelian subgroup of $\overline{\Sy}_\ell$ of order $\ell^3$.  We also have the commutation relations 
\begin{align*}
[x_1,x_2] &= x_3^{-1}x_4^{-1} \\
[x_1,x_3] &= x_4^{-2},
\end{align*}
which show that $x_1 \not \in  \langle x_2,x_3,x_4 \rangle$ because it does not commute with $x_2$ or $x_3$.  Therefore, $\langle x_1,x_2,x_3,x_4 \rangle = \overline{\Sy}_\ell$ since it must have maximal order $\ell^4$.  It is then a calculation to see that the element $x_4$ lies in the center of $\overline{\Sy}_\ell$.  By the commutation relations, the center of $\overline{\Sy}_\ell$ is $\langle x_4 \rangle$, while the center of $\overline{\Sy}_2$ is $\langle x_3 ,x_4 \rangle$.  We can simplify the generating set even more: when $\ell$ is odd we have $\overline{\Sy}_\ell = \langle x_1, x_2 \rangle$, while $\overline{\Sy}_2 = \langle x_1,x_2,x_4 \rangle$.  

For $\ell \geq 5$, $\overline{\Sy}_\ell$ has exponent $\ell$, while for $\ell \in \lbrace 2,3 \rbrace$ $\overline{\Sy}_\ell$ has exponent $\ell^2$ (this is a special case of a general fact about the Sylow subgroups of classical groups in defining characteristic \cite[Cor.~0.5]{testerman}).  In these two special cases the $\ell$-Sylow subgroups have isomorphism type $\overline{\Sy}_2 \simeq C_2 \times D_4$ and $\overline{\Sy}_3 \simeq C_3 \wr C_3$ (the wreath product of $C_3$ and $C_3$ with respect to a nontrivial permutation action).  We now seek an explicit description of the subgroups of $\overline{\Sy}_\ell$ of order $\ell^3$ which will be used in the proof of Proposition \ref{geq3sylow}.

To make the notation less cumbersome in the next lemma, we define the homomorphisms 
$$
\overline{\alpha}, \overline{\gamma}: \overline{\Sy}_\ell \to \Z/\ell
$$
 to be the ones induced by factoring $\alpha$ and $\gamma$ respectively through $\pi_{\ell}|_{\Sy_{\ell}} : \Sy_{\ell} \to \overline{\Sy}_{\ell}$; \emph{i.e.} given an element $g \in \overline{\Sy}_{\ell}$, the images $\overline{\alpha}(g)$ and $\overline{\gamma}(g)$ are its $(2,1)$-entry and its $(3,2)$-entry respectively.

\begin{lem} \label{jeff_lemma}
Let $\ell$ be an odd prime. There are exactly $\ell$ nonabelian subgroups of $\overline{\Sy}_\ell$ of order $\ell^3$, and these are given explicitly by
\begin{align} \label{extraspecial}
\langle x_1x_2^k, x_3 \rangle_{k=0,\dots,\ell-1}.
\end{align}
\end{lem}

\begin{proof}
It is routine to verify that the groups $\langle x_1x_2^k, x_3 \rangle_{k=0,\dots,\ell-1}$ are distinct, nonabelian, and of order $\ell^3$.  Additionally observe that $\langle x_2,x_3,x_4 \rangle$ coincides with the kernel of $\overline{\alpha}$ and is the unique abelian subgroup of $\overline{\Sy}_\ell$ of order $\ell^3$.  Therefore, if $H \subset \overline{\Sy}_\ell$ is a nonabelian subgroup of order $\ell^3$, then $\overline{\alpha}(H) \ne 0$.

Because $H$ is an $\ell$-group, it has nontrivial center.  Let $h \in H$ be such that $\overline{\alpha}(h) \ne 0$.  Then $h$ only commutes with powers of $x_4$.  Thus $H$ contains $\langle x_4 \rangle$. Since $h$ and $x_4$ each have order $\ell$ and commute, we have $\langle h,x_4 \rangle \simeq C_\ell^2$, hence $H$ contains another element $g$ so that $H = \langle g,h,x_4 \rangle$.  Since $H$ is nonabelian we must have $H = \langle g,h \rangle$ since otherwise $\langle g,h \rangle \simeq C_\ell^2$ and then $H$ would be elementary abelian.

Next, observe that $H \cap \langle x_2,x_3,x_4 \rangle$ has order $\ell^2$.  Since $\langle x_4 \rangle \subset H$, it must be the case that $H$ contains a subgroup $K$ of $\langle x_2,x_3\rangle$ of order $\ell$, whence we can write  $K = \langle x_2^cx_3^b \rangle$ for some $b,c$.  If $c \ne 0$, then we claim that the group generated by $h$, $x_4$, and $K$ is all of $\overline{\Sy}_\ell$.  This claim can be verified by checking that the commutator $[h, x_2^c x_3^b]$ equals a nontrivial power of $x_3$ times a power of $x_4$, thus ensuring that the group generated by $h$, $x_4$, and $K$ contains $\ker(\overline{\alpha})$; since $h \notin \ker(\overline{\alpha})$, this group coincides with $\overline{\Sy}_\ell$.  So in fact we can take $c = 0$ and have $K = \langle x_3 \rangle$.  

Thus, $G$ contains the order-$\ell^2$ subgroup $\langle x_3,x_4 \rangle$, and also contains the element $h$.  By multiplying $h$ by suitable powers of $x_3$ and $x_4$, we can take $h$ to be $x_1^ax_2^c$ for some $a,c$.  By raising $x_1^ax_2^c$ to a suitable power and re-multiplying by suitable powers of $x_3$ and $x_4$, we can take $h$ to be of the form $x_1x_2^k$, for some $k\in \lbrace 0,\dots,\ell-1 \rbrace$, as claimed.  
\end{proof}

\subsection{The case $\ell \geq 3$}  We now show that there are no counterexamples $G \subset \Sy_\ell$ when $\ell \geq 3$ by proving that if $G \in \Fix(\ell^2)$ then one of the homomorphisms $\alpha$, $\gamma$, $f$ is trivial on $G$ and applying Proposition \ref{latticeprop}.

\begin{proof}[Proof of Theorem \ref{jointsylowthm}(a)]
Let $\ell$ be an odd prime and $G \subset \Sy_\ell$ with $G \in \Fix(\ell^2)$.  We argue case by case based on the order of $\overline{G}$.

\

\noindent \underline{$\overline{G}$ has order $\ell^4$.} In this case $\overline{G} = \overline{\Sy}_\ell = \langle x_1,x_2 \rangle$. For $i \in \lbrace 1,2 \rbrace$ let $g_i \in G$ be any element such that $g_i \equiv x_i \pmod{\ell}$.  Observe that $\alpha$ and $\gamma$ are each non-zero on $g_1g_2$ and on $g_1g_2^2$, whence $f(g_1g_2) = f(g_1g_2^2) =0$; it then follows easily that $f(g_1) = f(g_2) = 0$.  Because $g_1$ and $g_2$ were chosen arbitrarily, it follows that $f$ is trivial on all of $G$. \\

\noindent \underline{$\overline{G}$ has order $\ell^3$.} There are $\ell+1$ subgroups of $\overline{\Sy}_\ell$ of order $\ell^3$.  One of these subgroups is elementary abelian and the remaining $\ell$ are nonabelian by Lemma \ref{jeff_lemma}.  If $\overline{G}$ is elementary abelian, then $\overline{G} = \langle x_2,x_3,x_4 \rangle$ and so $\alpha(G)=0$.  

For the nonabelian groups we appeal to the classification of Lemma \ref{jeff_lemma}.  Fix an index $k \in  \lbrace 0,\dots \ell-1 \rbrace$ and suppose $G$ is such that 
$$
\overline{G} = \langle x_1x_2^k, x_3 \rangle.
$$
If $k=0$ then visibly $\gamma(G) =0$.  If $k \ne 0$, then let $\mathsf{x}$ and $\mathsf{y}$ be any elements of $G$ such that $\mathsf{x} \equiv x_1x_2^k \pmod{\ell}$ and $\mathsf{y} \equiv x_3 \pmod{\ell}$.  Then $\alpha$ and $\gamma$ are nontrivial on $\mathsf{x}$ and $\mathsf{xy}$, so $f(\mathsf{x}) = f(\mathsf{xy})=0$ and thus $f(\mathsf{y}) = 0$.  Since $\mathsf{x}$ and $\mathsf{y}$ were chosen arbitrarily, $f(G) = 0$.\\

\noindent \underline{$\overline{G}$ has order $\ell^2$ or $\ell$.}  Since $G \in \Fix(\ell^2)$, for every $g \in G$ one of $\alpha(g),\gamma(g)$, or $f(g)$ must be trivial.  If $\overline{G}$ is cyclic and $g \in G$ is such that $\pi_\ell(g)$ generates $\overline{G}$, then whichever of $\alpha, \gamma$, or $f$ is trivial on $g$ must also be trivial on $G$.  This takes care of every group $G$ for which $\overline{G}$ has order $\ell$, as well as the special case of cyclic subgroups of $\overline{G}_3$ of order 9.  We will now assume $\overline{G}$ is elementary abelian of order $\ell^2$.

Suppose $\overline{G}$ contains an element $g$ on which $\overline{\alpha}$ and $\overline{\gamma}$ are both nontrivial.  Because $\overline{G}$ is abelian, and such $g$ 
 only commute with powers of $x_4$, then $\overline{G} = \langle x_4,g \rangle$.  But then $\overline{G}$ is also generated by $\langle gx_4,g \rangle$, and $\overline{\alpha}$ and $\overline{\gamma}$ are both nontrivial on these elements.  By the same reasoning as in the previous cases this implies $f(G)=0$.
 
If $\overline{G}$ contains an element $g\in \ker \overline{\alpha}$ then $g$ will not commute with any element of $\overline{\Sy}_\ell$ on which $\overline{\alpha}$ is nontrivial.  Therefore $\overline{G} \subset \ker \overline{\alpha}$ and so $G \subset \ker \alpha$.

If $\overline{G}$ contains an element $g \in \ker \overline{\gamma}$, let $h \in \overline{G}$ lie outside $\langle g \rangle$ so that $\overline{G}$ is generated by $g$ and $h$.  If $h \in \ker \overline{\gamma}$ then we are done.  If not, then both $g$ and $h$ must belong to $\ker \overline{\alpha}$ or else $g$ and $h$ would not commute and so $G \subset \ker \alpha$.
\end{proof}

\subsection{The case $\ell=2$} In contrast to the previous section, there do exist counterexamples $G \subset \Sy_2$, as we now show.

\begin{proof}[Proof of Theorem \ref{jointsylowthm}(b)]

We begin by noting, from the discussion above, that $\overline{\Sy}_2 = \langle x_1, x_2, x_4 \rangle$ and that each pair of these generators commutes except that we have $x_1 x_2 x_1^{-1} x_2^{-1} = x_3 x_4$; moreover, each of these generators commutes with $x_3$.  It is then straightforward to check that $\overline{\Sy}_2$ decomposes as a direct product of $\langle x_4 \rangle \cong C_2$ with $\langle x_1, x_2 \rangle \cong D_4$.  From evaluating $\overline{\alpha}$ and $\overline{\gamma}$ on generators it is clear that the order-$8$ elementary abelian group $\langle x_2, x_3, x_4 \rangle$ (resp. $\langle x_1, x_3, x_4 \rangle$) is contained in the kernel of $\overline{\alpha}$ (resp. $\overline{\gamma}$); moreover on checking orders we see that these containments are equalities.  It follows that the center of $\overline{\Sy}_2$ coincides with $Z := \langle x_3, x_4 \rangle = \ker(\overline{\alpha}) \cap \ker(\overline{\gamma})$.

We first prove that if $G \subset \Sy_2$ is a counterexample then we must have $\overline{G} = \overline{\Sy}_2$ or $\overline{G} \cong D_4$.  To show this, we start by claiming that if $\overline{G}$ is abelian then $\overline{G}$ cannot be a counterexample.  Suppose that $\overline{G}$ is an abelian subgroup of $\overline{\Sy}_2$.  Since neither $\overline{\alpha}$ nor $\overline{\gamma}$ can be trivial on $G$, there must exist (not necessarily distinct) elements $w, y \in \overline{G}$ such that $w \notin \ker(\overline{\alpha})$ and $y \notin \ker(\overline{\gamma})$.  If $w \in \ker(\overline{\gamma}) \smallsetminus \ker(\overline{\alpha})$, then $w \equiv x_1$ (mod $Z$), and the relations given above imply that $w$ cannot commute with anything not lying in $\ker(\overline{\gamma})$; this contradiction implies that $w \notin \ker(\overline{\alpha}) \cup \ker(\overline{\gamma})$.  By an analogous argument, we also have $y \notin \ker(\overline{\alpha}) \cup \ker(\overline{\gamma})$, and indeed, any element $g \in \overline{G} \smallsetminus Z$ must satisfy $g \notin \ker(\overline{\alpha}) \cup \ker(\overline{\gamma})$, \emph{i.e.} $\alpha(\tilde{g}) = \gamma(\tilde{g}) = 1$ for any $\tilde{g} \in G$ with $\pi_2(\tilde{g}) = g$.  Now if we assume that $G$ is a counterexample, for any $\tilde{g} \in G$ we must have $-\alpha(\tilde{g})^2 \gamma(\tilde{g}) f(\tilde{g}) \equiv 0$ (mod $2$) by (\ref{det_formula}) and therefore $f(\tilde{g}) = 0$ for each $\tilde{g} \in \pi_2^{-1}(\overline{G} \smallsetminus Z)$.  Then since $\overline{G} \smallsetminus Z$ clearly generates $\overline{G}$, we get that $f$ is trivial on $G$, thus contradicting our assumption and proving our claim.

We now assume that $\overline{G}$ is a proper nonabelian subgroup of $\overline{\Sy}_2$ (and therefore of order $8$) and show that it is isomorphic to $D_4$.  Note that since both $\ker(\overline{\alpha})$ and $\ker(\overline{\gamma})$ are elementary abelian $2$-groups, any order-$4$ element of $\overline{G}$ must lie in $\overline{\Sy}_2 \smallsetminus (\ker(\overline{\alpha}) \cup \ker(\overline{\gamma}))$.  By considering the quotient $\overline{\Sy}_2 / Z$ using the generators and relations given above, we see that any element of $\overline{\Sy}_2 \smallsetminus (\ker(\overline{\alpha}) \cup \ker(\overline{\gamma}))$ must be equivalent modulo $Z$ to $x_1 x_2$.  It follows that any two such elements commute, and so $\overline{G}$ has the property that any two of its order-$4$ elements commute.  Since the only nonabelian group of order $8$ with that property is $D_4$, we get $\overline{G} \cong D_4$ as claimed.

Now for any counterexample $G \subset \Sy_2$, we claim that $\langle G, H \rangle \in \Fix(4)$ for any subgroup $H \subset \Gamma(2) \cap \ker(f)$.  This follows directly from the formula (\ref{det_formula}) and the fact that replacing any element $g \in G$ with its translation by an element in $H$ clearly does not change $\alpha(g)$, $\beta(g)$, or $f(g)$.  Since we have $\langle G, H \rangle \supset G$, the fact that $G$ satisfies the lattice condition in Question \ref{kq} automatically implies that $\langle G, H \rangle$ satisfies it as well, and so $\langle G, H \rangle$ is also a counterexample.

Now that we have shown that any counterexample $G$ satisfies that $\overline{G}$ contains a subgroup isomorphic to $D_4$, we set out to prove the converse: that a counterexample subgroup $G \subset \Sy_2$ can be constructed satisfying that $\overline{G} = \overline{\Sy}_2$ or that $\overline{G}$ coincides with any given subgroup of $\overline{\Sy}_2$ which is isomorphic to $D_4$.  We start by letting $D \subset \overline{\Sy}_2$ be any subgroup isomorphic to $D_4$, generated by an order-$4$ element $\mathsf{x}$ and an order-$2$ element $\mathsf{y} \neq \mathsf{x}^2$.  Now suppose that $G = \langle \tilde{\mathsf{x}}, \tilde{\mathsf{y}}, \Gamma(2) \cap \ker(f) \rangle$, where $\tilde{\mathsf{x}}$ and $\tilde{\mathsf{y}}$ are elements of $\Sy_2$ lying in the inverse images $\pi_2^{-1}(\mathsf{x})$ and $\pi_2^{-1}(\mathsf{y})$ respectively and satisfying $f(\tilde{\mathsf{x}}) = 0$ and $f(\tilde{\mathsf{y}}) = 1$.  Then by construction we have $\overline{G} = \langle \mathsf{x}, \mathsf{y} \rangle = D$.  We now show that every element of $G$ satisfies the determinant condition required for $G$ to lie in $\Fix(4)$, for which we make use of the formula (\ref{det_formula}).  First of all, if $g \in G$ lies in $\pi_2^{-1}(\langle \mathsf{x} \rangle)$, then we clearly have $f(g) = 0$ and so $\det(g - 1) \equiv 0$ (mod $4$).  Now choose $g \in G \smallsetminus \pi_2^{-1}(\langle \mathsf{x} \rangle)$, so that $\pi_2(g) \in D \cong D_4$ has order $2$.  If we assume that $g \in \Sy_2 \smallsetminus (\ker(\overline{\alpha}) \cup \ker(\overline{\gamma}))$, then it is easily verified, using the fact that the only nontrivial commutator in $\overline{\Sy}_2$ lies in $Z$, that $\pi_2(g) \equiv x_1 x_2$ (mod $Z$) and so $\pi_2(g)^2 = x_3 x_4 \neq 1$, contradicting the fact that $\pi_2(g)$ has order $2$.  We therefore have $\pi_2(g) \in \ker(\overline{\alpha}) \cup \ker(\overline{\gamma})$.  We then get $\det(g - 1) \equiv 0$ (mod $4$) from the fact that $\alpha(g) = 0$ or $\gamma(g) = 0$.  It follows that $G \in \Fix(4)$.

Suppose that we replace $G$ with $\langle G, \tilde{x}_4 \rangle$ for some element $\tilde{x}_4 \in \Sy_2$ satisfying $\pi_2(\tilde{x}_4) = x_4$ and $f(\tilde{x}_4) = 0$.  We know from the group structure of $\overline{\Sy}_2$ that it is a direct product of $\langle x_4 \rangle$ and any of its subgroups isomorphic to $D_4$; therefore, we have $\overline{G} = \overline{\Sy}_2$.  Now given any $g \in G \smallsetminus \langle \tilde{\mathsf{x}}, \tilde{\mathsf{y}}, \Gamma(2) \cap H \rangle$, we have $g = g' x_4$ for some $g' \in \langle \tilde{\mathsf{x}}, \tilde{\mathsf{y}}, \Gamma(2) \cap H \rangle$.  We have already shown that $\det(g' - 1) \cong 0$ (mod $4$); now it is clear that $\det(g - 1) \equiv 0$ (mod $4$) also, using (\ref{det_formula}) and the fact that the homomorphisms $\alpha$, $\beta$, and $f$ each take the same value on $g'$ and $g' x_4$.  Thus, again we have $G \in \Fix(4)$.

Now, using the fact that the maps $\alpha$, $\gamma$, and $f$ are each nontrivial on $G$ in any of the above cases, we apply Lemma \ref{jeffsylowlemma} to get that $G$ is a counterexample.  We have thus proven the existence of counterexamples $G$ with $\overline{G} = \overline{\Sy}_2$ or $\overline{G} \cong D_4$.
\end{proof}

\subsection{Serre's Counterexample}  Because the reference \cite{serretokatz} does not appear in the literature, and because it was the genesis of this paper, we give a brief description of Serre's original counterexample.  Let $\ell = 2$ and consider the subgroup $H$ of $\Sy_2$ consisting of all $g$ such that 
$$
\alpha(g) + \gamma(g) + f(g) = 0.
$$
This ensures that the product $\alpha^2\gamma f$ is zero on $H$ (when $\ell=2$ we have $\alpha = \alpha'$) and hence that $H \in \Fix(4)$ by (\ref{det_formula}).  Now consider the elements
$$
g_1 = \left(
\begin{smallmatrix}
1&0&0&0 \\
1&1&0&0\\
0&0&1&0\\
0&0&-1&1
\end{smallmatrix}
\right) \qquad
g_2 = \left(
\begin{smallmatrix}
1&0&0&0 \\
0&1&0&0\\
0&1&1&0\\
0&0&0&1
\end{smallmatrix}
\right) \qquad
g_3 = \left(
\begin{smallmatrix}
1&0&0&2 \\
0&1&0&0\\
0&0&1&0\\
0&0&0&1
\end{smallmatrix}
\right).
$$
and set $A = g_1g_2$, $B = g_1g_3$, and $C = g_2g_3$.  Then one of  $\alpha$, $\gamma$, or $f$  is nontrivial on each of $A$, $B$, and $C$ and, additionally, $A,B$, and $C$ each belong to $H$.  This makes $H \subset \Sp_4(\Z_2)$ a counterexample by Lemma \ref{latticeprop} and Proposition \ref{jeffsylowlemma}.  We now replace $H$ by $\pi_4^{-1} (\pi_4(H))$, so that the image modulo 4 is the same, but $H$ now contains $\Gamma(4)$.
This enlarged group is then open in $\GSp_4(\Z_2)$, belongs to $\Fix(4)$, and does not stabilize any additional lattices; it is therefore a counterexample.  Since there exists an abelian surface over $\Q$ with full $2$-adic image $\GSp_4(\Z_2)$, one can enlarge the field of definition to produce an abelian surface over a number field with the desired mod-4 image which produces a counterexample to Question \ref{lq}.

\section{The Iwahori Subgroup} \label{iwahorisection}

Throughout this section, by fixing an appropriate symplectic basis of our free rank-$4$ $\Z_\ell$-module $\mathcal{T}$, we identify $\B_\ell$ with the subgroup of $\Sp_4(\Z_\ell)$ whose reduction modulo $\ell$ is the full subgroup of lower-triangular matrices.  Note that $\Sy_\ell \subset \B_\ell$.  If $G \subset \B_\ell$ also belongs to $\Fix(\ell^2)$, then the elements of $G$ can be explicitly described in terms of the maps $\alpha,\beta,\gamma,\delta,\epsilon$ as outlined in \S\ref{S2.1}.  The main result of this section is the following theorem, which classifies the counterexamples $G \subset \B_\ell$. Because $\B_2 = \Sy_2$, in this section we only consider primes $\ell \geq 3$. 

\begin{thm} \label{iwahori_prop}
Suppose that $\ell \geq 3$ and that $G \in \Fix(\ell^2)$ is a counterexample such that $G \subset \B_{\ell}$ but $G \not\subset \Sy_{\ell}$.  Then $G$ satisfies the following:

\begin{enumerate}
\item[(i)] $\alpha(G \cap \Sy_\ell) = 0$ and $\overline{(G \cap \Sy_{\ell})}$ has order $\ell$ or $\ell^2$; or 
\item[(ii)] $\gamma(G\cap \Sy_\ell) = 0$ and $\overline{(G \cap \Sy_{\ell})}$ has order $\ell$.
\end{enumerate}
In either case, for any $H \subset \Gamma(\ell)$, the subgroup of $\Sp_4(\Z / \ell^2)$ generated by $G$ and $H$ is also a counterexample.  

In particular, if $G$ is a maximal counterexample, then $\Gamma(\ell) \subset G$ and $\overline{G}$ has isomorphism type $\Z/\ell \times (\Z/\ell \rtimes (\Z/\ell)^\times)$ or  $\Z/\ell \rtimes (\Z/\ell)^\times$, depending on whether $\alpha(G \cap \Sy_\ell)= 0$ or $\gamma(G\cap \Sy_\ell) = 0$, respectively. 

Moreover, there do exist counterexamples satisfying (i) and counterexamples satisfying (ii).

\end{thm}

Observe that if $G \in \Fix(\ell^2)$ then $G \cap \Sy_\ell \in \Fix(\ell^2)$ as well.  By our work in \S\ref{sylowsection}, one of $\alpha$, $\gamma$, or $f$ must be trivial on $G \cap \Sy_\ell$.  Starting with $f$, we will consider the effect on $G$ of $\alpha$, $\gamma$, or $f$ being trivial on $G \cap \Sy_\ell$. 

\begin{lem} \label{f3kerlem}
Suppose $G \in \Fix(\ell^2)$, $G \subset \B_\ell$, and $f(G \cap \Sy_\ell)=0$.  Then $f(G)=0$ and therefore, by Proposition \ref{latticeprop}, $G$ is not a counterexample.
\end{lem}

\begin{proof}
The fact that $G \cap \Sy_\ell \vartriangleleft G$ and the hypothesis $f(G \cap \Sy_\ell) = 0$ together imply $f$ induces a homomorphism $G/G \cap \Sy_\ell \to \Z/\ell$.  But elements in $G/G \cap \Sy_\ell$ have order coprime to $\ell$, whence such a homomorphism is trivial and so is $f$.  
\end{proof}

In contrast to Lemma \ref{f3kerlem}, we do get counterexamples when $\alpha(G \cap \Sy_\ell)=0$ and when $\gamma(G \cap \Sy_\ell)=0$, as claimed in Theorem \ref{iwahori_prop}.  Our first step in each classification is to show that if $G$ is a maximal counterexample, then $\overline{(G \cap \Sy_\ell)}$ has order $\ell^2$ when $\alpha(G \cap \Sy_\ell)=0$ and order $\ell$ when 
$\gamma(G \cap \Sy_\ell) = 0$.

We start with a computation that will be used in both cases.  In order for $G$ to be a counterexample, $f: G \to \Z/\ell$ must be nontrivial and, therefore, surjective.  If $g \in G$ and $f(g) \ne 0$, our determinant formula (\ref{iwahori_det}) directly implies 
\begin{align} \label{detcond}
\gamma(g)\alpha(g)^2 + 
\frac{\beta(g) \alpha(g)(1-\epsilon(g)^2)}{\epsilon(g)}  +
 \frac{\delta(g)(1-\epsilon(g))^2}{\epsilon(g)} \equiv 0 \pmod{\ell}.
\end{align}

\begin{rmk}
Even though we will not need it for the work that follows, one can prove that if the mod-$\ell$ images of the entries of $g$ satisfy (\ref{detcond}), and if $\epsilon(g) \in (\Z/\ell)^\times$ has order $m$, then $g^m \equiv 1 \pmod{\ell}$.
\end{rmk}

Now we consider the effect of $\alpha$ and $\gamma$ being trivial on $G \cap \Sy_\ell$.  If either $\alpha(G\cap \Sy_\ell) = 0$ or $\gamma(G \cap \Sy_\ell)=0$, then $\overline{(G \cap \Sy_\ell)}$ cannot be the full $\ell$-Sylow subgroup of $\Sp_4(\Z/\ell)$.  We will now show, among other things, that $\overline{(G \cap \Sy_\ell)}$ cannot have order $\ell^3$ either.  To do this, we will argue separately for $\alpha$ versus $\gamma$.  Because neither $\alpha$ nor $\gamma$ extends to a homomorphism of $G$, our arguments will be different from those for Lemma \ref{f3kerlem}.

\begin{lem} \label{f0sylowprop}
Suppose $G \subset \B_\ell$ lies in $\Fix(\ell^2)$ and that $f|_{G}$ is nontrivial. Suppose further that $\alpha(G \cap \Sy_\ell) = 0$.  Then $\overline{(G \cap \Sy_\ell)}$ has order dividing $\ell^2$.
\end{lem}

\begin{proof}
Recall that $\ker \overline{\alpha} = \langle x_2,x_3,x_4 \rangle$ is the unique elementary abelian subgroup of $\overline{(G \cap \Sy_\ell)}$ of order $\ell^3$.  
Fix $g \in G \smallsetminus G \cap \Sy_\ell$ and suppose $\det(g-1) \equiv 0\pmod{\ell^2}$, so that either $f(g) =0$ or (\ref{detcond}) holds.

Let $s \in G \cap \Sy_\ell$. Then direct computation in coordinates reveals that
$$
\det(gs - 1) \equiv (\star) (f(g) + f(s)) \ell \pmod{\ell},
$$
where the expression $(\star)$ is given by
$$
\star = \frac{\alpha(g)^2 \gamma(s) + (1-\epsilon(g))^2 \delta(s) + 2\alpha(g) (1-\epsilon(g))\beta(s)}{\epsilon(g)}.
$$
Thus, for every $s \in G \cap \Sy_\ell$ we must have either 
\begin{align} \label{detcond2}
\alpha(g)^2 \gamma(s) + (1-\epsilon(g))^2 \delta(s) + 2\alpha(g) (1-\epsilon(g))\beta(s) \equiv 0 \pmod{\ell}
\end{align}
or $f(s) + f(g) \equiv 0\pmod{\ell}$.

For fixed $g$,  we claim that it is not the case that every $s \in \ker \alpha$ satisfies (\ref{detcond2}) or $f(s) + f(g) \equiv 0\pmod{\ell}$.  To see this, note that the subset 
$$
\lbrace ({\beta}(s),{\gamma}(s),{\delta}(s)) \rbrace_{s \in \ker \alpha \cap \Sy_\ell}   \subset (\Z/\ell)^3
$$
defines a 3-dimensional $\Fl$-vector space $\mathsf{E}$.  Indeed, it is easy to verify from the discussion in \S\ref{structure_of_sy} that the maps $\overline{\beta}$, $\overline{\gamma}$,  $\overline{\delta}: \overline{\Sy}_\ell \to \Z/\ell$ are homomorphisms and form a dual basis $\lbrace \overline{\beta}, \overline{\gamma}, \overline{\delta} \rbrace$ to the basis $\{x_2, x_3, x_4\}$ of the 3-dimensional $\Fl$-space $\ker  \overline{\alpha} \cap \overline{\Sy}_\ell$. 
Since $g$ is fixed and $\epsilon(g) \ne 1$, the congruence (\ref{detcond2}) defines a codimension-1 subspace $\mathsf{V}$ of $\mathsf{E}$.

Then every $s$ such that $ (\beta(s),\gamma(s),{\delta}(s)) $ lies outside $\mathsf{V}$
must have $f(s)  =-f(g)$.  If $f(g)=0$, then $f(s)=0$ for all $s \not \in \pi_\ell^{-1}(G \cap \Sy_\ell)$.  This implies $f(G) =0$ because the complement of $\pi_\ell^{-1}(G \cap \Sy_\ell)$ generates $G$, and contradicts the hypothesis $f|_{G}$ is nontrivial.  If $f(g) \ne 0$, then we have 
$$
f(s) = f(s^2) = -f(g),
$$
which is impossible since $f(s^2) = 2f(s)$. If follows that $\overline{(G \cap \Sy_\ell)}$ cannot have order $\ell^3$ and so must have order dividing $\ell^2$. 
\end{proof}

Lemma \ref{f0sylowprop} constrains the order of $\overline{(G \cap \Sy_\ell)}$ to be at most $\ell^2$.  We now show that counterexamples exist when the order equals $\ell^2$.  While it is possible that counterexamples may exist when the order of $\overline{(G \cap \Sy_\ell)}$ equals $\ell$, they would come from subgroups of the order-$\ell^2$ counterexamples.   Because of this, it will satisfy us to describe only the maximal counterexamples.

\begin{rmk}
In the extreme case where $\overline{(G \cap \Sy_\ell)}$ is trivial, then $G$ cannot be a counterexample, since $\overline{G}$ is then cyclic (if a generator fixes an order-$\ell^2$ submodule, then the entire group will fix the same).
\end{rmk}

\begin{prop} \label{iwahoricounterexampleprop1}
Fix an element $g \in \B_\ell \smallsetminus \Sy_\ell$ satisfying $\det(g-1)\equiv0\pmod{\ell^2}$.  Let $\mathcal{S}$ be the subgroup of $\ker \alpha$ satisfying (\ref{detcond2}) relative to the coordinates of $g$.  Then the subgroup $G$ of $\B_\ell$ generated by the element $g$ and the subgroups $\mathcal{S}$ and $\Gamma(\ell)$ is a counterexample.
\end{prop}

\begin{proof}
By hypothesis $\det(g-1) \equiv 0\pmod{\ell^2}$, and $\det(gs-1) \equiv 0\pmod{\ell^2}$ for all $s \in G \cap \Sy_\ell$ because of (\ref{detcond2}).  Therefore, the coset $g(G \cap \Sy_\ell)$ consists entirely of elements $\sigma$ satisfying $\det(\sigma-1) \equiv 0 \pmod{\ell^2}$.  We claim that this is enough to conclude that $G\in \Fix(\ell^2)$.  To see this we use the fact that $G/(G \cap \Sy_\ell)$ is cyclic, generated by $g(G \cap \Sy_\ell)$ and, for fixed $n \geq 1$, evaluate the expression ($\star$) of Lemma \ref{f0sylowprop} on elements $g^ns$:
$$
(\star)': \frac{\alpha(g^n)^2 \gamma(s) + (1-\epsilon(g^n))^2 \delta(s) + 2\alpha(g^n) (1-\epsilon(g^n))\beta(s)}{\epsilon(g^n)}.
$$
We have $\epsilon(g^n) = \epsilon(g)^n$ because $\epsilon: G \to (\Z/\ell)^\times$ is a homomorphism.   It is easy to show that 
\begin{align*}
\alpha(g^n) = \frac{1-\epsilon(g)^n}{1-\epsilon(g)} \alpha(g).
\end{align*}
Then applying the expressions for $\alpha(g^n)$ and $\epsilon(g^n)$ to $(\star)'$ and using (\ref{detcond2}), algebraic manipulation reveals that ($\star)' = 0$.  Therefore, every coset $g^n(G \cap \Sy_\ell)$ consists of $\sigma$ with $\det(\sigma - 1) \equiv 0\pmod{\ell^2}$ and so $G \in \Fix(\ell^2)$.

To see that $G$ is a counterexample, we apply Propositions \ref{latticeprop} and \ref{jeffsylowlemma}.  Our assumption that $\Gamma(\ell) \subset G$ means that $f|_{G}$ is nontrivial, satisfying Proposition \ref{latticeprop}(a).  If $\alpha|_G$ and $\gamma|_G$ are nonzero, then Proposition \ref{jeffsylowlemma} shows that the only proper $G$-stable lattices we need to check for quotients with trivial $G$-action are $L_0$, $L_1$,$ L_2$, and $L_3$.  But any pair of these with relative index $\ell^2$ visibly has nontrivial $G$-action due to the nontriviality of $\epsilon$.  There is one exceptional case to check by hand.

If $\alpha(G) = 0$, then $\delta(s) =0$ for all $s \in G \cap \Sy_\ell$ by (\ref{detcond2}). If, in addition,  $\gamma(G) \ne 0$, then an argument in the vein of the proof of Proposition \ref{latticeprop} shows that the only new $G$-stable lattices to include among $L_0$, $L_1$, $L_2$, and $L_3$ are:
\begin{itemize}
\item the lattice $\widetilde{L}_1$ generated by $\ell \mathcal{T}$ and $e_3$, and
\item the lattice $\widetilde{L}_2$ generated by $\ell \mathcal{T}$ and $e_1$, $e_3$, and $e_4$.
\end{itemize}
The nontriviality of $\epsilon$ again shows that the action on any quotient of order $\ell^2$ is nontrivial.

If $\gamma(G) =0$, (in particular $\gamma(s) =0$  for all $s \in G \cap \Sy_\ell$), then (\ref{detcond2}) imposes an additional linear condition on the entries of $\overline{G \cap \Sy_\ell}$ and so $G \cap \Sy_\ell$ has order dividing $\ell$.  Since we only classify maximal counterexamples in this proposition, we can safely omit this case.  

This completes the classification of maximal counterexamples and finishes the proof.

\end{proof}

We can produce counterexamples that are as large as possible within the constraints of Proposition \ref{iwahoricounterexampleprop1}, as the following example shows.

\begin{example} Suppose $\alpha(g) =0$ so that ($\star$)' simplifies to 
$$
(1-\epsilon(g)^n)^2\delta(s).
$$
for all $s$. In particular, $(1-\epsilon(g)^n)^2\delta(s)$ must equal 0 for all powers of $n$, including those for which $\epsilon(g)^n \ne 1$, whence $\delta(s) = 0$ for all $s$.  The maximal subgroup satisfying all of these conditions is then seen to be the preimage in $\Sp_4(\Zl)$ of the group
$$
\left(\begin{smallmatrix}
1&0&0&0 \\ 0 & \epsilon & 0 & 0 \\ \beta & \gamma & 1/\epsilon & 0 \\0 & \beta \epsilon & 0 & 1 
\end{smallmatrix}
\right) \subset \Sp_4(\Z/\ell).
$$
where $\beta,\gamma \in \Z/\ell$ and $\epsilon \in (\Z/\ell)^\times$. 
\end{example}

Finally we consider the case where $\gamma(G \cap \Sy_\ell) = 0$.  Similar to when $\alpha(G \cap \Sy_\ell)=0$, we will show that if $G$ is a counterexample, then $\overline{(G \cap \Sy_\ell)}$ cannot have order $\ell^3$; in fact, we will show that $\overline{(G \cap \Sy_\ell)}$ must have order $\ell$.  

\begin{lem} \label{f1sylowprop}
Suppose $G \subset \B_\ell$ lies in $\Fix(\ell^2)$ and that $f|_{G}$ is nontrivial. Suppose further that $\gamma(G \cap \Sy_\ell) = 0$.  Then $\overline{(G \cap \Sy_\ell)}$ has order dividing $\ell$.
\end{lem}

\begin{proof}
The proof strategy is nearly identical to that of Lemma \ref{f0sylowprop}.  Fix $g \in G \smallsetminus G \cap \Sy_\ell$ with $\det(g-1) \equiv 0\pmod{\ell}$.  For all $s \in G \cap \Sy_\ell$ we must have $\det(gs-1) \equiv 0\pmod{\ell}$ and a direct calculation reveals that 
\begin{align*}
\det(gs-1) &\equiv \left( {\Huge{\mathbf{\star \star}}} \right)(f(g) + f(s))\ell,
\end{align*}
where the expression $(\mathbf{\star} \star)$ is given by 
\begin{align*}
&\gamma(g) \alpha(s)^2 + 
2(\beta(g) (1-\epsilon(g) ) + \alpha(g)  \gamma(g) ) \alpha(s) + \\
& 2\alpha(g) (1/\epsilon(g)  - 1) \beta(s) + 
(1/\epsilon(g)  - 1)\alpha(s) \beta(s) +  
(\epsilon(g)  + 1/\epsilon(g)  -2)\delta(s).
\end{align*}
Therefore, for every $s \in G \cap \Sy_\ell$ it must be the case that either 
\begin{align} \label{f1detcond}
&\gamma(g) \alpha(s)^2 + 
2(\beta(g) (1-\epsilon(g) ) + \alpha(g)  \gamma(g) ) \alpha(s) + \\
& 2\alpha(g) (1/\epsilon(g)  - 1) \beta(s) + 
(1/\epsilon(g)  - 1)\alpha(s) \beta(s) +  
(\epsilon(g)  + 1/\epsilon(g)  -2)\delta(s)  =0 \nonumber
\end{align}
or $f(s) + f(g) =0$.   

Not every triple $({\alpha}(s),{\beta}(s),{\delta}(s)) \in (\Z/\ell)^3$ satisfies (\ref{f1detcond}) (for example, $(0,0,1)$ does not), and those that do not need not satisfy $f(s) + f(g)=0$ by the same reasoning in the proof of Lemma \ref{f0sylowprop}.  Therefore, the group $\overline{(G \cap \Sy_\ell)}$ cannot have order $\ell^3$, and hence has order dividing $\ell^2$.  We will now show that the order must in fact divide $\ell$.

While the group $\ker \overline{\gamma}$ is not elementary abelian, any subgroup of order dividing $\ell^2$ is.  We will show that the set 
$$
\lbrace ({\alpha}(s),{\beta}(s),{\delta}(s)) \in \Fl^3\rbrace
$$
taken over all $s \in G \cap \Sy_\ell$ which satisfy (\ref{f1detcond}) does not contain a 2-dimensional linear space.  Suppose it did.  Let $s \in G \cap \Sy_\ell$ so that $s^2 \in G \cap \Sy_\ell$ as well.   Apply the condition  (\ref{f1detcond}) to $s^2$ and subtract twice the relation  (\ref{f1detcond}) applied to $s$ to obtain the new condition
\begin{align} \label{new1}
\alpha(s) \left(\gamma(g) \alpha(s) +(1/\epsilon(g) - 1) \beta(s) \right) =0.
\end{align}
If $\alpha(s)=0$ then substituting into (\ref{f1detcond}) shows 
\begin{align} \label{new2}
2\alpha(g)(1/\epsilon(g) - 1) \beta(s) + (\epsilon(g) + 1/\epsilon(g) -2)\delta(s)  =0
\end{align}
whence the linear space is at most 1-dimensional.

On the other hand if $\gamma(g) \alpha(s) +(1/\epsilon(g) - 1) \beta(s) =0$, then substituting into (\ref{f1detcond}) additionally shows that 
\begin{align} \label{new3}
2\beta(g)(1-\epsilon(g))\alpha(s) + (\epsilon(g) + 1/\epsilon(g) -2)\delta(s)  =0
\end{align}
as well, whence the linear space is at most 1 dimensional.

In all cases, we see that $\overline{G \cap S}$ has order dividing $\ell$, which completes the proof of the lemma.
\end{proof}

We now show that there exist counterexamples $G \subset \B_\ell$ where $\overline{(G \cap \Sy_\ell)}$ has order $\ell$.

\begin{prop} \label{iwahoricounterexampleprop2}
Fix an element $g \in \B_\ell \smallsetminus \Sy_\ell$ with $\det (g-1) \equiv 0\pmod{\ell^2}$.  Let $\mathscr{S}$ be any subgroup of $\ker \gamma$ satisfying (\ref{f1detcond}) such that $\overline{\mathscr{S}}$ has order $\ell$.  Then the subgroup $G$ of $\B_\ell$ generated by $g$, $\mathscr{S}$, and $\Gamma(\ell)$ is a counterexample.
\end{prop}

\begin{proof}
The proof of Lemma \ref{f1sylowprop} showed that there are two ways for $\overline{G \cap \Sy_\ell}$ to have order $\ell$; see equation (\ref{new1}), which yields the two cases (\ref{new2}) and (\ref{new3}).  We will consider these case by case. \\

\noindent \emph{Case 1.} Suppose $\alpha(s) = 0$ for all $s \in G \cap \Sy_\ell$. Then by (\ref{f1detcond}) we have 
\begin{align} \label{proofcond1}
2\alpha(g)(1-\epsilon(g))\beta(s) + (\epsilon(g)+1/\epsilon(g)-2)\delta(s) = 0
\end{align}
and so $g(G \cap \Sy_\ell)$ consists entirely of elements $\sigma$ such that $\det(\sigma -1) \equiv 0 \pmod{\ell^2}$.  Fix a positive integer $n >1$ such that $g^n \not \in G \cap \Sy_\ell$  and consider the coset $g^n(G \cap \Sy_\ell)$.  Because 
$$
\gamma(g^n) = \frac{(1-\epsilon(g)^{2n})}{(1-\epsilon(g)^2) \epsilon(g)^{n-2}} \gamma(g), 
$$
the determinant condition $\det(g^ns-1) \equiv 0 \pmod{\ell^2}$, under the assumption that $\alpha(s) = 0$ and  making the substitution $g \mapsto g^n$ in (\ref{f1detcond}) reduces to
$$
2\alpha(g^n)(1-\epsilon(g^n))\beta(s) + (\epsilon(g^n)+1/\epsilon(g^n)-2)\delta(s) \equiv 0 \pmod{\ell},
$$
which simplifies to 
$$
2\alpha(g)\beta(s)(1-\epsilon(g)^n)^2 \frac{(1-1/\epsilon(g)^{(n-1)})}{1-\epsilon(g)} \equiv 0\pmod{\ell}
$$
for all $s \in G \cap \Sy_\ell$.  If $\beta(s) = 0$ for all $g \in \Sy_\ell$, then by (\ref{proofcond1}) we have $\delta(s) = 0$ for all $s \in \Sy_\ell$ as well, and so $\overline{G \cap \Sy_\ell}$ is trivial, violating the hypothesis that it have order $\ell$.  We also assume that $\epsilon(g^n)$  and $\epsilon(g^{n-1})$ are  nontrivial, and thus we are left with $\alpha(g)=0$.

If $\alpha(g) = 0$ then $\alpha(G)= 0$, and then it is immediate to check that (\ref{f1detcond}) is satisfied for all nontrivial cosets $g^n(G \cap \Sy_\ell)$.  \\

\noindent \emph{Case 2.} Here we assume 
\begin{align} \label{proofcond2}
\gamma(g)a(s) + (1/\epsilon(g) -1)\beta(s)=2\beta(g)(1-\epsilon(g))\alpha(s) + (\epsilon(g)+1/\epsilon(g)-2)\delta(s) = 0
\end{align}
for all $s \in G \cap \Sy_\ell$. Now we proceed in an identical fashion to the previous case to determine conditions for an arbitrary coset $g^n(G \cap \Sy_\ell)$ to consist of elements $\sigma$ with $\det(\sigma -1) \equiv 0\pmod{\ell^2}$.  If $\gamma(g) \ne 0$, then a similar argument to the one above shows that we are forced to take $\alpha(s) = 0$ for all $s$.  But the linearity conditions then show $\beta(s) = \delta(s) = 0$ for all $s$ as well, whence $\overline{G \cap \Sy_\ell}$ is trivial, a contradiction.  On the other hand, if $\gamma(g)=0$, then $\gamma(G) = 0$; and then a  straightforward argument, similar to the one above, shows (substituting $g \mapsto g^n$ in the formula (\ref{proofcond2})) that $\det(g^ns -1) \equiv 0\pmod{\ell^2}$ for all $s \in G\cap \Sy_\ell$ and so $G \in \Fix(\ell^2)$.  

We see that either $\alpha|_{G} = 0$ or $\gamma|_{G}=0$, so we cannot apply Proposition \ref{jeffsylowlemma} directly.  In the case $\alpha|_{G}=0$ but $\gamma|_{G} \ne 0$, then the only lattices to check in addition to the $L_i$ are the $\widetilde{L}_1$ and $\widetilde{L}_2$ of the proof of Proposition \ref{iwahoricounterexampleprop1}.  Similarly, due to the nontriviality of $\epsilon$, there do not exist quotients of order $\ell^2$ with trivial $G$-action.

If $\gamma|_{G} =0$ but $\alpha|_{G} \ne 0$, then the only additional stable lattice to check is the lattice $\widetilde{L}_{3}$ generated by $\ell \mathcal{T}$ and the elements $e_2$ and $e_4$.  As in all other cases, the nontriviality of $\epsilon$  means that none of the quotients of order $\ell^2$ have trivial $G$-action.

This completes the proof.   
\end{proof}

As with Proposition \ref{iwahoricounterexampleprop1}, we can use Proposition \ref{iwahoricounterexampleprop2} to produce maximal counterexamples.  That is, we can find $G \in \Fix(\ell^2)$ such that $\Gamma(\ell) \subset G$ and such that $\overline{G}$ has order $(\ell-1)\ell$.  The following example has $G \cap \Sy_\ell \subset \ker \gamma$ and is distinct from the ones classified by Proposition \ref{iwahoricounterexampleprop1}.

\begin{example} Let $G \subset B_\ell$ be the full preimage of the group
\begin{align*}
\left(
\begin{smallmatrix}
1&0&0&0 \\
\alpha & \epsilon &0&0 \\
0 &0 &1/\epsilon &0 \\
0 & 0 &-\alpha/\epsilon &1 
\end{smallmatrix}
\right) \subset \Sp_4(\Z/\ell),
\end{align*}
where $\alpha \in \Z/\ell$ and $\epsilon \in (\Z/\ell)^\times$.  One can check that this group falls into the classification of the counterexamples given above. 
\end{example}

\section{Subgroups with an irreducible 2-dimensional factor} \label{parabolic_section}

Now we suppose that the subgroup $G \subset \Sp(\mathcal{T})$ is such that the semisimplification of the action of $\overline{G}$ on $\mathcal{T} / \ell \mathcal{T}$ contains an irreducible 2-dimensional factor.  Throughout this section, by fixing an appropriate symplectic basis $\{e_1, e_2, e_3, e_4\}$ of our free rank-$4$ $\Z_\ell$-module $\mathcal{T}$ (in which we require that $\overline{e}_4$ be fixed by all of $\overline{G}$), we identify $\Pa_\ell$ with a particular subgroup of $\Sp_4(\Z_\ell)$ whose reduction is lower block-triangular.  We recall the maps $\alpha, \beta, \beta', \alpha' : \Pa_\ell \to \Z/\ell$, as well as their induced maps $\overline{\alpha}, \overline{\beta}, \overline{\alpha}', \overline{\beta}' : \overline{\Pa}_\ell \to \Z/\ell$ defined in \S\ref{S2.1}.  Each element of $\overline{\Pa}_\ell$ is a lower-block-diagonal matrix which fixes $\overline{e}_4$ and whose middle block is a $2 \times 2$ submatrix reflecting how that operator acts on the component corresponding to the span of $\{\overline{e}_2, \overline{e}_3\}$ in the semisimplification of $\mathcal{T} / \ell \mathcal{T}$.  There is therefore a homomorphism $\overline{\pi} : \overline{\Pa}_\ell \to \SL_2(\Z/\ell)$ given by sending a matrix in $\overline{\Pa}_\ell$ to its middle block, which is a matrix in $\SL_2(\Z/\ell)$.  Composing this with $\pi_{\ell}$ gives us a homomorphism $\pi : G \to \SL_2(\Z/\ell)$.

When the image under $\pi$ of a subgroup $G \subset \Pa_\ell$ is reducible, via an appropriate change of symplectic bases of $\mathcal{T}$, it can be simultaneously conjugated to a group of lower-triangular matrices in $\SL_2(\F_{\ell})$ (without affecting the block-diagonal structure of $\overline{G}$), and therefore we are in the situation dealt with in \S\ref{sylowsection} and \S\ref{iwahorisection}.  In this section, we are concerned with the case that $\overline{\pi}(\overline{G})$ is an irreducible subgroup of $\SL_2(\F_{\ell})$, \emph{i.e.} there is no nontrivial subspace of $\F_{\ell}^2$ fixed by the whole group $\pi(G)$.

We now show how the vector $(\beta', \alpha')(g) = (\beta'(g), \alpha'(g)) \in \Fl^2$ is determined by $(\alpha, \beta)(g)$ and $\pi(g)$.  The group $\SL_2(\F_{\ell})$ injects into $\overline{\Pa}_\ell$ as the subgroup of all block-diagonal matrices whose first and last blocks are trivial; this injection $\SL_2(\F_{\ell}) \hookrightarrow \overline{\Pa}_\ell$ is a section of the surjective map $\overline{\pi}$.  Given any matrix $g \in \Pa_\ell$, we may multiply its reduction $\pi_\ell(g) \in \overline{\Pa}_\ell$ on the right by the block-diagonal matrix in $\overline{\Pa}_\ell$ corresponding to the image of $\pi(g)^{-1}$ to get a matrix $x \in \overline{\Pa}_\ell$ lying in $\ker(\overline{\pi})$.  We have seen in \S\ref{background} that then we have $(\overline{\beta}', \overline{\alpha}')(x) = (\overline{\beta}, -\overline{\alpha})(x)$.  One then checks directly through the operation of matrix multiplication that $(\alpha, \beta)(g) = (\overline{\alpha}, \overline{\beta})(\pi_\ell(g)) = (\overline{\alpha}, \overline{\beta})(x)$ and that we have the formula 
\begin{equation} \label{beta'}
(\beta', \alpha')(g) = \left(\begin{matrix} \beta(g) & -\alpha(g) \end{matrix}\right) \pi(g).
\end{equation}

We are now ready to present the main result of this section, which states that under the hypotheses of this section there are no counterexamples for $\ell \geq 3$ and which roughly classifies the counterexamples that exist for $\ell = 2$.  For notational convenience we switch to using $\Fl$ for $\Z/\ell$ to emphasize the fact that we are doing linear algebra rather than considering an image of reduction modulo $\ell$. 

\begin{thm} \label{parabolic}

Let $G \subset \Sp(\mathcal{T})$ be a counterexample satisfying that $\pi(G) \subset \SL_2(\F_{\ell})$ is an irreducible subgroup.  Then we have $\ell = 2$; \emph{i.e.} there are no counterexamples satisfying the above property when $\ell \geq 3$.  When $\ell = 2$, we have that $G \subset \Sp(\mathcal{T})$ satisfies either
\begin{itemize}
\item[(i)] $\overline{G}_2 \cong \pi(G) \times C_2$; or
\item[(ii)] $\overline{G}_2 \cong \pi(G) \cong \SL_2(\F_2)$.
\end{itemize}

In case (i), for any $H \subset \Gamma(2) \cap \ker(f)$, the subgroup of $\Sp(\mathcal{T})$ generated by $G$ and $H$ is also a counterexample.  In case (ii), for any $H \subset \ker(\pi_2)$, the subgroup of $\Sp(\mathcal{T})$ generated by $G$ and $H$ is also a counterexample (in particular, the full inverse image $\pi_2^{-1}(\overline{G}_2)$ is a maximal counterexample).

Moreover, there do exist counterexamples satisfying (i) and counterexamples satisfying (ii).

\end{thm}

The rest of this section is dedicated to proving the above theorem.  Throughout the following arguments, we will freely use the fact that if a subgroup $G \subset \Sp(\mathcal{T})$ is a counterexample, then $f$ must be nontrivial on $G$, by Proposition \ref{latticeprop}(a).

\subsection{Restricting the possible images of counterexamples}

In this subsection, we will show that under the assumption of an irreducible $2$-dimensional factor which was established at the beginning of this section, a counterexample $G$ must satisfy $\overline{G} \cong \pi(G) \times C_{\ell}$ or $\overline{G} \cong \pi(G)$.  We first need to provide some basic results concerning the properties of the classical groups $\SL_2(\F_{\ell})$ and their irreducible subgroups, as in the below proposition.

For the statement below and the arguments given throughout the rest of the section, recall that a \textit{unipotent} operator $x$ is one satisfying $(x - 1)^n = 0$ for some $n \geq 1$.  In our situation where $x$ belongs to $\SL_2(\F_{\ell})$ for some prime $\ell$, this is equivalent to satisfying that $(x - 1)^2 = 0$; that $x - 1$ is non-invertible; that $x$ fixes some nontrivial vector $v \in \F_{\ell}^2$; or that the only eigenvalue of $x$ is $1$.

In what follows, an \textit{irreducible} subgroup of $\SL_2(\F_{\ell})$ is one which acts irreducibly on $\F_{\ell}^2$.

\begin{prop} \label{classicalgroups}

Let $\ell$ be a prime.  The following facts hold.

a) (i) If $\ell \geq 5$, then there is no nontrivial homomorphism from $\SL_2(\F_{\ell})$ to $\Z / \ell$.

\ \ (ii)The only normal subgroup of $\SL_2(\F_3)$ of index $3$ is the subgroup $Q_8 \lhd \SL_2(\F_3)$ coinciding with the subset of all elements whose orders are not divisible by $3$ and which is isomorphic to the quaternion group; there are thus only two nontrivial homomorphisms from $\SL_2(\F_3)$ to $\Z / 3$, both having kernel $Q_8$.

\ (iii) Since $\SL_2(\F_2)$ is isomorphic to the symmetric group $S_3$, the only nontrivial homomorphism from $\SL_2(\F_2)$ to $\Z / 2$ is the one whose kernel is the order-$3$ subgroup $A_3 \lhd \SL_2(\F_2)$ corresponding to the alternating group.

b) An element of $\SL_2(\F_{\ell})$ is unipotent if and only if it has order dividing $\ell$.

c) The order of any proper irreducible subgroup of $\SL_2(\F_{\ell})$ is not divisible by $\ell$; thus there is no nontrivial homomorphism from a proper irreducible subgroup of $\SL_2(\F_{\ell})$ to $\Z / \ell$.

d) Assume that $\ell \geq 3$, and let $H \subsetneq \SL_2(\F_{\ell})$ be a proper irreducible subgroup.  The group $\SL_2(\F_{\ell})$ is generated by set of non-unipotent matrices lying in $\SL_2(\F_{\ell}) \smallsetminus H$.

e) If $\ell \geq 3$, each irreducible subgroup of $\SL_2(\F_{\ell})$ has nontrivial center.

\end{prop}

\begin{proof}

If $\ell \geq 3$, any homomorphism from $\SL_2(\F_{\ell})$ to $\Z / \ell\Z$ must kill the scalar $-1$, since the order of this element is never divisible by $\ell$.  Such a homomorphism therefore factors through the projective linear group $\SL_2(\F_{\ell}) / \{\pm 1\}$.  According to \cite[\S IV.3.4, Lemma 1]{serre}, this group is simple as long as $\ell \geq 5$.  Such a homomorphism must therefore be trivial, proving part (a)(i).  The statements of (a)(ii) and (a)(iii) are evident from direct verification.

Parts (c) and (b) are precisely the statement of \cite[Theorem XI.2.2]{lang} (see also \cite[\S IV.3.2, Lemma 2]{serre} and the unnamed statement appearing right before it in \cite[\S XI.2]{lang} respectively.

Now assume that $\ell \geq 3$, and let $H \subsetneq \SL_2(\F_{\ell})$ be a proper irreducible subgroup.  Consider the subset $S \subset \SL_2(\F_{\ell})$ consisting of all matrices $x$ such that $-x$ is nontrivial and unipotent.  Since each operator in $S$ has $-1$ as its only eigenvalue, there are no unipotent matrices in $S$.  Moreover, given any element $x \in S$, since $\ell$ is odd and $-x \neq 1$ is unipotent and so has order $\ell$, we have $x^{\ell + 1} = (-x)^{\ell + 1} = -x$.  This proves both that $S \cap H = \varnothing$ (because otherwise $H$ would contain the unipotent matrix $-x$ for each $x \in S$; it follows from parts (b) and (c) that this contradicts the fact that $H$ is proper and irreducible) and that $S$ generates the set of all unipotent matrices in $\SL_2(\F_{\ell})$, which are well known to generate all of $\SL_2(\F_{\ell})$.  Thus, part (d) is proved.

Now retaining our assumption that $\ell \geq 3$, the group $\SL_2(\F_{\ell})$ itself has nontrivial center since it contains the scalar $-1$.  Let $N \subsetneq \SL_2(\F_{\ell})$ be a proper irreducible subgroup.  By part (c), the order of $N$ is not divisible by $\ell$ and so we may apply \cite[Theorem XI.2.3]{lang} to get that $N / (N \cap \{\pm 1\})$ is isomorphic to a dihedral group or to $A_4$, $S_4$, or $A_5$.  One verifies through straightforward computation that the only element of order $2$ in $\SL_2(\F_{\ell})$ is the scalar $-1$.  It follows that if $N$ has even order, then $N$ has nontrivial center.  We therefore assume that $N$ has odd order.  Then we have that $N \cong N / (N \cap \{\pm 1\})$ itself must be an odd-order subgroup of a dihedral group or of $A_4$, $S_4$, or $A_5$.  We claim that the only odd-order subgroups of these groups are abelian, thus proving that $N$ still has nontrivial center.  Indeed, the only odd-order elements of a dihedral group lie in its index-$2$ cyclic subgroup and thus can only generate a cyclic subgroup, while we see by looking at the orders of $A_4$, $S_4$, and $A_5$ that their odd-order subgroups must have order dividing $15$, and all such groups are abelian.  Thus, part (e) is proved.

\end{proof}

\begin{lem} \label{coboundary}

Let $\ell$ be any prime and $G$ be any group in $\Fix(\ell^2)$.

a) For each nontrivial element $g \in G$ with $f(g) \neq 0$, there exists a vector $w_g = ((w_g)_1, (w_g)_2) \in \F_{\ell}^2$ such that $(\alpha, \beta)(g) = \pi(g).w_g - w_g$.  If $f(g) \neq 0$ and $\pi(g)$ is not unipotent, then $\pi_{\ell}(g)$ fixes the vector $(-1, (w_g)_1, (w_g)_2, 0) \in \F_{\ell}^4$ and we have the formula 
$$\delta(g) = \left(\begin{matrix} \beta(g) & -\alpha(g) \end{matrix}\right) (\pi(g) - 1)^{-1} \left(\begin{matrix} \alpha(g) \\ \beta(g) \end{matrix}\right).$$

b) Suppose that the maps $\alpha$ and $\beta$ both vanish on the subgroup $G \cap \ker(\pi)$.  Then there exists a vector $w = (w_1, w_2) \in \F_{\ell}^2$ (depending only on $G$) such that $(\alpha(g), \beta(g)) = \pi(g).w - w$ for every element $g \in G$.  For those elements $g \in G$ such that $f(g) \neq 0$ and $\pi(g)$ is not unipotent, we have that $\pi_{\ell}(g)$ fixes the vector $(-1, w_1, w_2, 0) \in \F_{\ell}^4$.

\end{lem}

\begin{proof}

Let $g \in G$ be an element such that $f(g) \neq 0$.  Since $G \in \Fix(\ell^2)$, it follows from Remark \ref{fix} that $\pi_{\ell^2}(g) \in \Sp_4(\Z / \ell^2)$ fixes a submodule of $\mathcal{T} / \ell^2 \mathcal{T}$ of order $\ell^2$.  It is already clear that $\pi_{\ell^2}(g)$ fixes the mod-$\ell^2$ image of $\ell e_4$; it must therefore be the case that $\pi_{\ell^2}(g)$ fixes a vector $u \in \mathcal{T} / \ell^2 \mathcal{T}$ such that $\ell e_4$ modulo $\ell^2$ is $\ell u$ (that is, $u = (\ell u_1, \ell u_2, \ell u_3, u_4)$ with $u_4 \neq 0$) or that there is a vector $v \in \mathcal{T} / \ell^2 \mathcal{T}$ with $\ell v \neq \ell \pi_{\ell^2}(e_4)$ such that $g$ fixes $\ell v$.  The first case is impossible, as one verifies easily that the first entry of $g.u$ equals the first entry of $u$ plus $f(g) u_4 \neq 0$.  We therefore have a vector $v \in \mathcal{T} / \ell^2 \mathcal{T}$ such that $\ell g.v = \ell v$, or equivalently, such that the image modulo $\ell$ of $v$ is fixed under multiplication by $\pi_{\ell^2 \to \ell}(g)$.  Write $\overline{v} = (v_1, v_2, v_3, v_4) \in \F_\ell^4$ for the image modulo $\ell$ of $v$.  We observe that the second and third entries of $\pi_\ell(g) \overline{v}$ are given by the vector $v_1 (\alpha(g), \beta(g)) + \pi(g). (v_2, v_3)$.  The fact that $\pi_\ell(g).\overline{v} = \overline{v}$ now implies 
\begin{equation}
(v_2, v_3) = v_1 (\alpha(g), \beta(g)) + \pi(g).(v_2, v_3).
\end{equation}
  If $v_1 \neq 0$, then it follows that $(\alpha, \beta)(g) = -v_1^{-1}(\pi(g) \, (v_2, v_3) - (v_2, v_3))$, and we get the first claim of part (a) when we take $w_g = -v_1^{-1}(v_2, v_3)$.  We therefore assume that $v_1 = 0$.  In this case the above equation implies that $(v_2, v_3)$ is invariant under multiplication by $\pi(g)$ (in particular, this implies that $\pi(g)$ is unipotent).  It follows from the above discussion that the final entry of $\pi_2(g).\overline{v}$ is equal to 
$$\left( \begin{matrix} \beta & -\alpha \end{matrix} \right) \pi(g) \left( \begin{matrix} v_2 \\ v_3 \end{matrix} \right) + v_4 = \left( \begin{matrix} \beta & -\alpha \end{matrix} \right) \left( \begin{matrix} v_2 \\ v_3 \end{matrix} \right) + v_4.$$
  Since we have $\pi_2(g).\overline{v} = \overline{v}$, it follows that $\left( \begin{smallmatrix} \beta & -\alpha \end{smallmatrix} \right) \left( \begin{smallmatrix} v_2 \\ v_3 \end{smallmatrix} \right) = 0$.  It is now clear that the vector $(\beta, -\alpha) \in \F_{\ell}^2$ must be a scalar multiple of $(v_3, -v_2)$, and so $(\alpha, \beta)(g)$ is a scalar multiple of $(v_2, v_3)$.  Then we take $w_g$ to be any vector in the subspace $\langle \overline{e}_2, \overline{e}_3 \rangle$ which is \textit{not} a scalar multiple of $(v_2, v_3)$.  Since the operator $\pi(g)$ is not the identity, it cannot also fix $w_g$.  It is now easily verified that $\pi(g).w_g - w_g$ is a nontrivial scalar multiple of $(v_2, v_3)$ and thus also of $(\alpha(g), \beta(g))$; after replacing $w_g$ with a suitable multiple of itself, we even get $\pi(g).w_g - w_g = (\alpha(g), \beta(g))$, and the first claim of part (a) follows.

Now suppose that $f(g) \neq 0$ and $\pi(g)$ is not unipotent.  We have seen above that $\pi_{\ell}(g)$ fixes a vector $(v_1, v_2, v_3, v_4) \in \F_{\ell}^4$ and that we must have $v_1 \neq 0$, because otherwise $\pi(g)$ would be unipotent.  We have shown that in this case, we may take $w_g = -v_1^{-1}(v_2, v_3)$.  Since $\pi(g)$ fixes both $-v_1^{-1}(v_1, v_2, v_3, v_4) = (-1, (w_g)_1, (w_g)_2, -v_1^{-1}v_4)$ and $(0, 0, 0, 1)$, we get the claim that $(-1, (w_g)_1, (w_g)_2, 0)$ is fixed by $\pi(g)$.  Now the final entry of $\pi(g) (-1, (w_g)_1, (w_g)_2, 0)$ is given by 
\begin{equation}
0 = -\delta(g) + \beta'(g)(w_g)_1 + \alpha'(g)(w_g)_2 = -\delta(g) + \left(\begin{matrix} \beta'(g) & \alpha'(g) \end{matrix}\right) \left(\begin{matrix} (w_g)_1 \\ (w_g)_2 \end{matrix}\right).
\end{equation}
  Since from the discussion at the start of this section we have $(\beta', \alpha')(g) = \left( \begin{matrix}\beta(g) & -\alpha(g) \end{matrix} \right) \pi(g)$ and we have shown above that $(\alpha, \beta)(g) = (\pi(g) - 1).w_g$, we get 
\begin{equation}
\delta(g) = \left(\begin{matrix} \beta(g) & -\alpha(g) \end{matrix}\right) \pi(g) (\pi(g) - 1)^{-1} \left(\begin{matrix} \alpha(g) \\ \beta(g) \end{matrix}\right).
\end{equation}
  Now one sees that the above is equivalent to the formula claimed in part (a) by noting that $\pi(g) (\pi(g) - 1)^{-1} = (\pi(g) - 1)^{-1} + 1$ and $\left(\begin{matrix} \beta(g) & -\alpha(g) \end{matrix}\right) \left(\begin{matrix} \alpha(g) & \beta(g) \end{matrix}\right)^{\top} = 0$.

We now assume the hypothesis of part (b), which implies that the map $(\alpha, \beta) : G \to \F_{\ell}^2$ induces a map $(\hat{\alpha}, \hat{\beta}) : \pi(G) \to \F_{\ell}^2$.  We observe directly from multiplying matrices that we have the identity 
\begin{equation} \label{cocycle}
(\hat{\alpha}, \hat{\beta})(xy) = (\hat{\alpha}, \hat{\beta})(x) +  x.(\hat{\alpha}, \hat{\beta})(y)
\end{equation}
 for any $x, y \in \pi(G)$.  The map $(\hat{\alpha}, \hat{\beta})$ is therefore a cocycle with respect to the obvious action of $\pi(G)$ on $\F_{\ell}^2$.  The claim of part (b) is equivalent to saying that $(\hat{\alpha}, \hat{\beta})$ is also a coboundary, so it suffices to prove that the first group cohomology of $\pi(G)$ with coefficients in the $\pi(G)$-module $\F_{\ell}^2$ is trivial.  We first assume that we do not have $\ell = 2$ and $\pi(G) = \SL_2(\F_2)$ and prove the vanishing of the first group cohomology by appealing to Sah's Lemma \cite[Lemma VI.10.2]{lang}, which implies as an immediate corollary that if a group $A$ acting on a module $M$ has a central element $x$ such that $x - 1$ acts as an automorphism on $M$, then the first group homomology $H^1(A, M)$ vanishes.  In our case, we need to show that $\pi(G)$ has a central element $x$ such that the operator $x - 1$ is invertible.  Either we have $\ell = 2$ and $\pi(G) = A_3$ (which is a nontrivial abelian group), or we have $\ell \geq 3$ and then Proposition \ref{classicalgroups}(e) implies that $\pi(G)$ has a nontrivial central element.  In either case, choose an element $x \neq 1$ lying in the center of $\pi(G)$.  Since a unipotent operator cannot lie in the center of an irreducible subgroup of $\SL_2(\F_{\ell})$, the operator $x - 1$ must be invertible, and we have proved part (b) except in the exceptional case that $\ell = 2$ and $\pi(G) = \SL_2(\F_2)$.  In this case, we consider the values that $(\overline{\alpha}, \overline{\beta})$ takes on the elements $u_1 := \left(\begin{smallmatrix} 1 & 1 \\ 0 & 1 \end{smallmatrix}\right)$ and $u_2 := \left(\begin{smallmatrix} 1 & 0 \\ 1 & 1 \end{smallmatrix}\right)$, noting that together these elements generate $\SL_2(\F_2)$.  It follows from part (a) that there exist scalars $c_1, c_2 \in \F_2$ such that $(\hat{\alpha}, \hat{\beta})(u_1) = c_1 (1, 0)$ and $(\hat{\alpha}, \hat{\beta})(u_2) = c_2 (0, 1)$.  Taking $w = (c_2, c_1) \in \F_2^2$, we get the desired statement.
\end{proof}

\begin{lem} \label{alphabeta0}

If $G$ is a counterexample, then the maps $\alpha$ and $\beta$ vanish on the subgroup $G \cap \ker(\pi)$.

\end{lem}

\begin{proof}

Suppose that there is an element $g \in G \cap \ker(\pi)$ with $(\alpha, \beta)(g) \neq (0, 0)$.  Using the identity in (\ref{cocycle}), it is easy to verify that for any $h \in G$ we have $(\alpha(h g h^{-1}), \beta(h g h^{-1})) = \pi(h)(\alpha, \beta)(g)$.  Since $\pi(G)$ is irreducible, there is some $h \in G$ such that the set $\{\pi(h).(\alpha, \beta)(g), (\alpha, \beta)(g)\}$ is linearly independent.  We note that the map $(\alpha, \beta)$ is a homomorphism when restricted to $G \cap \ker(\pi)$ thanks to the identity (\ref{cocycle}).  It follows that given any vector $(\alpha_0, \beta_0) \in \F_{\ell}^2$, there is an element $g \in G \cap \ker(\pi)$ with $(\alpha, \beta)(g) = (\alpha_0, \beta_0)$.  We shall show that for any $g \in G \cap \ker(\pi)$ such that $(\alpha, \beta)(g) \neq (0, 0)$, we have $f(g) \neq 0$.  Under the assumption that such an element $g$ exists, this implies an obvious contradiction and thus will prove the statement in the lemma.

In order to prove our claim that $f(g) \neq 0$ for any $g \in G \cap \ker(\pi)$ such that $(\alpha, \beta)(g) \neq (0, 0)$, we consider the cases $\pi(G) = \SL_2(\F_{\ell})$ and $\pi(G) \subsetneq \SL_2(\F_{\ell})$ separately.  We first assume that $\pi(G) = \SL_2(\F_{\ell})$.  Assume that there exists an element $h \in G \cap \ker(\pi)$ with $(\alpha, \beta)(h) \neq (0, 0)$ and $f(h) = 0$.  Let $y \in \F_{\ell}^2$ be a vector which is not a scalar multiple of $(\alpha, \beta)(h)$.  Then $\pi(G)$ contains a nontrivial unipotent operator $u$ which fixes $y$, so that $uw - w$ is a scalar multiple of $y$ for each $w \in \F_{\ell}^2$.  There exists some $g \in \pi^{-1}(u) \subset G$ with $f(g) \neq 0$, because otherwise, the fact that any nontrivial unipotent operator in $\SL_2(\F_{\ell})$ normally generates all of $\SL_2(\F_{\ell})$ implies that $f$ is trivial on all of $G$.  Now Lemma \ref{coboundary}(a) implies that $(\alpha, \beta)(g)$ is a scalar multiple of $y$.  We have $\pi(hg) = u$ and $f(hg) \neq 0$, while the identity (\ref{cocycle}) implies that $(\alpha, \beta)(hg) = (\alpha, \beta)(h) + (\alpha, \beta)(g)$, which is not a scalar multiple of $y$, thus contradicting Lemma \ref{coboundary}(a).

We now assume that $\pi(G) \subsetneq \SL_2(\F_{\ell})$.  Assume again that there exists an element $h \in G \cap \ker(\pi)$ with $(\alpha, \beta)(h) \neq (0, 0)$ and $f(h) = 0$.  Since the set of all such elements is clearly closed under conjugation and group multiplication, this implies that in fact for every vector $(\alpha_0, \beta_0) \in \F_{\ell}^2$ there is an element $h \in G \cap \ker(\pi)$ with $(\alpha, \beta)(h) = (\alpha_0, \beta_0)$ and $f(h) = 0$.  There exists an element $g \in G \smallsetminus \ker(\pi)$ with $f(g) \neq 0$, because otherwise $f$ would be trivial on $G$.  Since $\pi(G)$ is a proper irreducible subgroup of $\SL_2(\F_{\ell})$, we know that $\pi(g)$ is not unipotent by Lemma \ref{coboundary}(b), (c), and so we may apply Lemma \ref{coboundary}(a) to get the formula given there for $\delta(g)$.  We may do the same to get a formula for $\delta(hg)$ for any $h \in G \cap \ker(\pi) \cap \ker(f)$ since then $\pi(hg) = \pi(g)$ and $f(hg) = f(g) \neq 0$.  Using the previously noted fact that $(\alpha, \beta)(hg) = (\alpha, \beta)(h) + (\alpha, \beta)(g)$ along with the easily verified fact that $\delta(hg) = \delta(h) + \delta(g)$, we get the below general formula for $\delta(hg)$.
\begin{equation} \label{delta}
\delta(hg) = \big(\left(\begin{matrix} \beta(h) & -\alpha(h) \end{matrix}\right) + \left(\begin{matrix} \beta(g) & -\alpha(g) \end{matrix}\right)\big) (\pi(g) - 1)^{-1} \Big(\left(\begin{matrix} \alpha(h) \\ \beta(h) \end{matrix}\right) + \left(\begin{matrix} \alpha(g) \\ \beta(g) \end{matrix}\right)\Big)
\end{equation}
  We now expand the above formula, use the easily verified fact that $\delta(hg) = \delta(h) + \delta(g)$, and subtract the formula for $\delta(g)$ from (\ref{delta}) to get 
\begin{equation} \label{delta2} \begin{split}
\delta(h) = \left(\begin{matrix} \beta(h) & -\alpha(h) \end{matrix}\right) (\pi(g) - 1)^{-1} \left(\begin{matrix} \alpha(g) \\ \beta(g) \end{matrix}\right) &+ \left(\begin{matrix} \beta(g) & -\alpha(g) \end{matrix}\right) (\pi(g) - 1)^{-1} \left(\begin{matrix} \alpha(h) \\ \beta(h) \end{matrix}\right) \\
&+ \left(\begin{matrix} \beta(h) & -\alpha(h) \end{matrix}\right) (\pi(g) - 1)^{-1} \left(\begin{matrix} \alpha(h) \\ \beta(h) \end{matrix}\right).
\end{split} \end{equation}
  We now use the fact that the final term on the right-hand side of (\ref{delta2}) is in some sense ``quadratic" while the other terms in (\ref{delta2}) are ``linear" in order to derive a contradiction.  More precisely, we consider the cases when $\ell \geq 3$ and $\ell = 2$ separately as follows.  We note in either case that $(\alpha, \beta, \delta)(h^2) = 2(\alpha, \beta, \delta)(h)$ for any $h \in G \cap \ker(\pi)$.  If $\ell \geq 3$, then choose an element $h \in G \cap \ker(\pi)$ such that $(\alpha(h), \beta(h))$ is not an eigenvector of $\pi(g) - 1$, which ensures that 
\begin{equation} \label{quadratic}
\left(\begin{matrix} \beta(h) & -\alpha(h) \end{matrix}\right) (\pi(g) - 1)^{-1} \left(\begin{matrix} \alpha(h) \\ \beta(h) \end{matrix}\right) \neq 0.
\end{equation}
  Then applying the formula (\ref{delta2}) to $h^2$ and subtracting (\ref{delta}), we get 
\begin{equation}
2\left(\begin{matrix} \beta(h) & -\alpha(h) \end{matrix}\right) (\pi(g) - 1)^{-1} \left(\begin{matrix} \alpha(h) \\ \beta(h) \end{matrix}\right) = 0,
\end{equation}
 which contradicts (\ref{quadratic}).  Now if $\ell = 2$, we deduce from the relations given in \S\ref{sylowsection} that $\alpha$, $\beta$, and $\delta$ are all homomorphisms when restricted to $G \cap \ker(\pi)$.  Noting that $\pi(g) \in \{\left(\begin{smallmatrix} 1 & 1 \\ 1 & 0 \end{smallmatrix}\right), \left(\begin{smallmatrix} 0 & 1 \\ 1 & 1 \end{smallmatrix}\right)\}$, we now compute 
$$\left(\begin{matrix} \beta(h) & -\alpha(h) \end{matrix}\right) (\pi(g) - 1)^{-1} \left(\begin{matrix} \alpha(h) \\ \beta(h) \end{matrix}\right) = \alpha(h)^2 + \alpha(h)\beta(h) + \beta(h)^2.$$
  Choosing elements $h_1, h_2 \in G \cap \ker(\pi)$ such that $(\alpha, \beta)(h_1) = (1, 0)$ and $(\alpha, \beta)(h_2) = (0, 1)$, putting $h = h_1 + h_2$ into (\ref{delta2}) and then subtracting the formula (\ref{delta2}) for $h_1$ and for $h_2$ yields the desired contradiction.
\end{proof}

\begin{cor}

Let $G \subset \Sp(\mathcal{T})$ be a counterexample satisfying that $\pi(G) \subset \SL_2(\F_\ell)$ is an irreducible subgroup.  We then have $\overline{G} \cong \pi(G) \times C_{\ell}$ or $\overline{G} \cong \pi(G)$.

\end{cor}

\begin{proof}

Clearly $\overline{G}$ is an extension of $\pi(G)$ by $\overline{G}_\ell \cap \pi_\ell(\ker(\pi))$.  Since Lemma \ref{alphabeta0} says that the maps (homomorphisms) $\alpha$ and $\beta$ vanish on the latter, we get that $\overline{G} \cap \pi_{\ell}(\ker(\pi)) \subset \langle x_4 \rangle \cong C_{\ell}$, where $x_4$ is the element defined in \S\ref{sylowsection}.  Moreover, it follows from the discussion there that $x_4$ commutes with everything in $\pi_{\ell}(\ker(\pi)) \subset \Sp_4(\mathcal{T} / \ell \mathcal{T})$, which directly implies the desired statement.
\end{proof}

\subsection{The nonexistence of counterexamples for $\ell \geq 3$}

We assume throughout this subsection that $\ell \geq 3$ and proceed to prove the first statement of Theorem \ref{parabolic}.

\begin{proof}[Proof of Theorem \ref{parabolic} for $\ell \geq 3$]

We first consider the case that $\overline{G} \cong \SL_2(\F_{\ell}) \times C_{\ell}$.  In this case, we claim that $f$ is trivial on $G \cap \ker(\pi_{\ell})$.  To see this, assume that $f$ is nontrivial on $G \cap \ker(\pi_{\ell})$.  Let $g \in G$ be an element such that $\pi(g)$ is not unipotent, and let $h \in G$ be an element such that $\pi_{\ell}(h) = x_4$.  Then after possibly translating $g$ or $h$ by an element of $G \cap \ker(\pi_{\ell}) \smallsetminus \ker(f) \neq \varnothing$, we get that $f(g) = f(hg) \neq 0$.  Since $h \in \ker(\pi)$, we have already seen that $(\alpha, \beta)(hg) = (\alpha, \beta)(g)$; meanwhile, one verifies in a straightforward manner that $\delta(hg) = \delta(g) + 1$.  Then putting $hg$ into the formula given by Lemma \ref{coboundary}(a) yields the desired contradiction.

The fact that the homomorphism $f$ is trivial on $G \cap \ker(\pi_{\ell})$ implies that it induces a homomorphism $\overline{f} : \overline{G}_\ell \to \Z / \ell\Z$.  Write $\overline{G}_\ell = S \times \langle x_4 \rangle$, where $S$ is a component isomorphic to $\pi(G)$.  It follows from Proposition \ref{classicalgroups}(a) that $\overline{f}$ is trivial on $S$ except in the case that $\ell = 3$ and $S \cong \SL_2(\F_3)$.  Assume for the moment that we are not in that exceptional case.  Then we must have $\overline{f}_3(x_4) \neq 0$, because otherwise $f$ would be trivial on $G$.  Let $g, h \in G$ be elements such that $\pi_{\ell}(h) = x_4$, $\pi_{\ell}(g) \in S$, and $\pi(g)$ is not unipotent.  Then $f(hg) = 1 \in \Z / \ell$, so that we may apply Lemma \ref{coboundary}(a) to get the formula given there for $\delta(hg)$.  Since $\ell \geq 3$, we have $f(h^2 g) = 2 \neq 0 \in \Z / \ell$, and since again $\delta(h^2 g) = \delta(hg) + 1$ and $(\alpha, \beta)(h^2 g) = (\alpha, \beta)(hg)$, the formula given by Proposition \ref{coboundary}(a) applied to $\delta(h^2 g)$ implies a contradiction.  It follows that there is no counterexample for these cases.

We now assume that $\ell = 3$ and $S \cong \SL_2(\F_3)$.  We claim that there is a subgroup $S' \subset \overline{G}_3$ with $\overline{G}_3 = S' \times \langle x_4 \rangle$ such that $\overline{f}$ is trivial on $S'$, so that the above argument works in this case also by replacing $S$ with $S'$.  If $\overline{f}$ is trivial on $S$, then we take $S' = S$ and are done, so we assume that $\overline{f}$ is not trivial on $S$.  Note that there exists $g \in S$ with $\overline{f}(g) \neq 0$ such that $\pi(g)$ is not unipotent, since by Proposition \ref{classicalgroups}(a)(ii) we have $\ker(\overline{f}) \cap \pi(G) \subset Q_8$ and certainly there are non-unipotent matrices in $\SL_2(\F_3) \smallsetminus Q_8$.  Now again letting $h \in G$ be an element such that $\pi_3(h) = x_4$, if we assume that $\overline{f}(h) = 0$ and apply a similar argument as was used above to the formula for $\delta(hg)$ given by Proposition \ref{coboundary}(a), we get a contradiction.  Therefore, the homomorphism $\overline{f}$ is nontrivial on $\langle x_4 \rangle$ and we may define $S'$ to be $\{x_4^{-\overline{f}_3(y) \overline{f}(x_4)^{-1}}y \ | \ y \in S\}$; it is easy to check that $S' \subset \overline{G}_3$ is a subgroup contained in the kernel of $\overline{f}$ and satisfying $\overline{G} = S' \times \langle x_4 \rangle$.  We have thus shown that there are no counterexamples in the case that $\overline{G} \cong \SL_2(\F_{\ell}) \times C_{\ell}$.

We now consider the case that $\overline{G}_\ell \cong \pi(G)$.  First suppose that there is an element $h \in G \cap \ker(\pi_{\ell})$ with $f(h) \neq 0$.  We shall show that $\overline{G}$ fixes a $2$-dimensional subspace of $\mathcal{T} / \ell \mathcal{T}$, and that therefore $G$ is not a counterexample by Proposition \ref{latticeprop}(c).  Let $x \in \overline{G} \cong \pi(G)$ be any non-unipotent operator and lift it to an element $g \in G$ with $\pi_{\ell}(g) = x$.  If $f(g) = 0$, we let $g' = hg$, and we let $g' = g$ otherwise, so that $f(g') \neq 0$.  Now by Lemma \ref{coboundary}(b), there is a vector $w = (w_1, w_2) \in \F_{\ell}^2$ such that $x = \pi_{\ell}(g')$ fixes the vector $(-1, w_1, w_2, 0) \in \F_{\ell}^4$.  Now Proposition \ref{classicalgroups}(d) says that the subset of non-unipotent operators in $\pi(G)$ generates $\pi(G)$; it follows that the whole group $\overline{G}$ fixes $(-1, w_1, w_2, 0)$.  Since the group $\overline{G}_\ell$ also fixes $(0, 0, 0, 1)$, it fixes the $2$-dimensional subspace generated by these two vectors and so $G$ is not a counterexample.

Now suppose that the homomorphism $f$ is trivial on $G \cap \ker(\pi_{\ell})$, so that $f$ induces a homomorphism $\overline{f} : \overline{G} \to \Z / \ell$.  We know that $\overline{f}$ cannot be trivial on $\overline{G} \cong \pi(G)$, because otherwise $f$ would be trivial on all of $G$, so we are left with the only possibility being that $\ell = 3$ and $\overline{G}_3 \cong \SL_2(\F_3)$ with the induced homomorphism $\overline{f} : \SL_2(\F_3) \to \Z / 3$ being a surjection whose kernel is $Q_8$.  There exist non-unipotent operators in $\SL_2(\F_3) \smallsetminus \ker(\overline{f}_3)$ which generate all of $\SL_2(\F_3)$ by Proposition \ref{classicalgroups}(d).  Then the argument proceeds in a similar fashion: we know from Lemma \ref{coboundary}(b) that there is a vector $w = (w_1, w_2) \in \F_3^2$ such that $x \in \overline{G}_3 \smallsetminus \ker(\overline{f}_3)$ fixes the vector $(-1, w_1, w_2, 0) \in \F_{\ell}^4$ if $x$ is not unipotent; since the set of such elements generates all of $\overline{G}_3$, we get that the whole group $\overline{G}_3$ fixes the $2$-dimensional subspace spanned by $\{(-1, w_2, w_3, 0), (0, 0, 0, 1)\}$ and therefore $G$ is not a counterexample.

\end{proof}

\subsection{Classifying counterexamples for $\ell = 2$}

In this subsection, we assume that $\ell = 2$ and finish the proof of Theorem \ref{parabolic}.  We first present the following useful lemma.

\begin{lem} \label{paraboliclattices}

Let $S \subset \overline{\Sy}_2$ be the subgroup isomorphic to $S_3$ which fixes the subspace $\langle \overline{e}_1, \overline{e}_4 \rangle \subset \mathcal{T} / \ell \mathcal{T}$ and acts as $\SL_2(\F_2)$ on its complement subspace $\langle \overline{e}_2, \overline{e}_3 \rangle \subset \mathcal{T} / \ell \mathcal{T}$, and let $x_4 \in \overline{\Sy}_2$ be the element defined in \S\ref{sylowsection}.  Suppose that $G \subset \Sy_2$ is a subgroup satisfying one of the following:

\begin{enumerate}
\item[(i)] $\overline{G} = \langle x_4 \rangle \times S$;
\item[(ii)] $\overline{G} \subset \langle x_4 \rangle \times S$ is the subgroup isomorphic to $A_3 \times C_2$; or 
\item[(iii)] $\overline{G} \subset \langle x_4 \rangle \times S$ is the subgroup given by $\{(x, \phi(x)) \in S \times \langle x_4 \rangle \ | \ x \in S\}$, where $\phi : S \to \langle x_4 \rangle$ is the unique surjective homomorphism.
\end{enumerate}

Then the only nontrivial $G$-invariant sublattices of $\mathcal{T}$ which properly contain $\ell\mathcal{T}$ are $M_1 := \langle \ell\mathcal{T}, e_1, e_4 \rangle$, $M_2 := \langle \ell\mathcal{T}, e_2, e_3 \rangle$, $L_1$ and $L_3$, where the $L_i$'s are as in (\ref{latticedef}).

\end{lem}

\begin{proof}

We write $\overline{M}_1$, $\overline{M}_2$, $\overline{L}_1$, and $\overline{L}_3$ for the subspaces of $\mathcal{T} / 2\mathcal{T}$ given by the quotients by $2\mathcal{T}$ of $M_1$, $M_2$, $L_1$, and $L_3$ respectively.  The statement of the lemma is equivalent to saying that the only proper, nontrivial $\overline{G}_2$-invariant subspaces of $\mathcal{T} / 2\mathcal{T}$ are $\overline{M}_1$, $\overline{M}_2$, $\overline{L}_1$, and $\overline{L}_3$.

Note that in cases (i), (ii), and (iii) of the statement, the only proper nontrivial $\overline{G}_2$-invariant subspace of $\overline{M}_1$ is $\langle \overline{e}_4 \rangle = L_3$, while $\overline{M}_2$ has no proper, nontrivial $\overline{G}_2$-invariant subspaces.  Choose any vector $v \in \mathcal{T} / \ell\mathcal{T}$, and let $\overline{L} \in \mathcal{T} / \ell\mathcal{T}$ be the smallest $\overline{G}_2$-invariant subspace containing $v$.  Since the vector space $\mathcal{T} / \ell\mathcal{T}$ can be decomposed as the direct sum $\overline{M}_1 \oplus \overline{M}_2$, we may write $v = m_1 + m_2$ for some vectors $m_1 \in \overline{M}_1$ and $m_2 \in \overline{M}_2$.

We claim that $\overline{L} = \overline{L}_1 \oplus \overline{L}_2$, where $\overline{L}_i$ is the smallest $\overline{G}_2$-invariant subspace of $\overline{M}_i$ containing $m_i$ for $i = 1, 2$.  If $m_1 = 0$ or $m_2 = 0$, this is trivially true, so we assume that $m_1, m_2 \neq 0$.  It is straightforward to check for cases (i), (ii), and (iii) given in the statement that given any element $m'_2 \in \overline{M}_2 \smallsetminus \{0, m_2\}$, there is an element of $\overline{G}_2$ which sends $v = m_1 + m_2$ to $m_1 + m'_2$; taking their difference, we get that $0 \neq m_2 - m'_2 \in \overline{M}_2$ lies in $L$.  Since in all of the cases in the statement, $\overline{G}_2$ acts irreducibly on $\overline{M}_2$, we get $\overline{L}_2 = \overline{M}_2 \subsetneq L$.  Now since $m_2 \in \overline{L}$, we have $v - m_2 = m_1 \in \overline{L}$, so that $\overline{L}_1$ by definition is contained in $\subset \overline{L}$.  We therefore have $\overline{L}_1 \oplus \overline{L}_2 \subset \overline{L}$, and since $\overline{L}_1 \oplus \overline{L}_2$ is a $\overline{G}_2$-invariant subspace containing $v$, this inclusion of subspaces is in fact an equality, whence our claim.  The statement of the lemma follows now from the observation that $\overline{L}_1 = \overline{M}_2 \oplus \overline{L}_3$.
\end{proof}

\begin{proof}[Proof of Theorem \ref{parabolic} for $\ell = 2$]

We first consider the case that $\overline{G} \cong \pi(G) \times C_2$.  Here we have that $f$ is trivial on $G \cap \ker(\pi_2)$ by the exact same argument we used in \S\ref{sylowsection} under the case that $\ell = 3$ and $\overline{G} \cong \pi(G) \times C_{\ell}$, so we again have an induced homomorphism $\overline{f} : \overline{G} \to \Z / 2$.  As before, we write $\overline{G} = S \times \langle x_4 \rangle$, where $S \cong \pi(G)$.  We also have that $f(x_4) \neq 0$ by the same argument as was used under Case 1.  Fix an element $h \in G$ with $\pi_2(h) = x_4$.  We now claim that $G$ is a counterexample if and only if for each element $g \in G$ such that $\pi(g)$ is not unipotent, we have either (i) $f(g) = 1$ and $\delta(g)$ is equal to the expression in terms of $(\alpha, \beta)(g)$ in the formula given in Lemma \ref{coboundary}(a), or (ii) $f(g) = 0$ and $\delta(g)$ is \textit{not} equal to the expression in terms of $(\alpha, \beta)(g)$ in the formula given in Lemma \ref{coboundary}(a).  Indeed, if such an element $g \in G$ satisfies neither (i) nor (ii), then either $g$ or $hg$ clearly violates the conclusion of Lemma \ref{coboundary}(a), which contradicts the fact that $G \in \Fix(4)$.  Suppose conversely that either (i) or (ii) holds for all such elements $g \in G$; we will show that now $G$ is a counterexample.  For each element $g \in G$ such that $\pi(g)$ is unipotent and fixes a nontrivial vector $v = (v_1, v_2) \in \F_2^2$, we have that $\pi_2(g)$ fixes the $2$-dimensional subspace of $\mathcal{T} / 2\mathcal{T}$ spanned by $\{(0, v_1, v_2, 0), (0, 0, 0, 1)\}$.  Meanwhile, for each $g \in G$ such that $\pi(g) = \pi(hg)$ is not unipotent, Lemma \ref{coboundary}(b) says that for some vector $w = (w_1, w_2) \in \F_2^2$, either $\pi_2(g)$ or $\pi_2(hg)$ fixes the vector $(-1, w_1, w_2, 0) \in \mathcal{T} / 2\mathcal{T}$.  An easy computation shows that if $\pi_2(g)$ fixes $(-1, w_1, w_2, 0)$, then $\pi_2(hg)$ does not fix $(-1, w_1, w_2, 0)$ but $hg$ does fix $(2, 2, 2, 1) \in \mathcal{T} / 2\mathcal{T}$, and vice versa.  Moreover, the subspace of $\mathcal{T} / 2\mathcal{T}$ fixed by $\pi_2(g)$ contains the vector $(0, 0, 0, 1)$ but can have dimension at most $2$ (otherwise $\pi(g) - 1$ would be non-invertible so that $\pi(g)$ would be unipotent).  It follows that there is no $2$-dimensional subspace of $\mathcal{T} / 2\mathcal{T}$ fixed by the whole group $\overline{G}$.  Since of course $f$ is not trivial on $G$, we get that $G$ is a counterexample by Proposition \ref{latticeprop}(c) and by a quick check of quotients of the $G$-stable sublattices provided by Lemma \ref{paraboliclattices}.  Now it is clear that for any subgroup $H \subset \ker(\pi_2) \cap \ker(f)$, the group generated by the subgroups $G$ and $H$ is also a counterexample since multiplying by elements in $H$ will not affect their images under $\alpha$, $\beta$, $\delta$, or $f$.

We now consider the case that $\ell = 2$ and $\overline{G} \cong \pi(G)$.  We first eliminate the possibility that $\overline{G} \cong \pi(G) = A_3$.  Indeed, if we have $\overline{G} \cong \pi(G) = A_3$, then the group $\overline{G} \cong \pi(G)$ is cyclic and there is some element $g \in G$ with $f(g) \neq 0$ such that $\pi_2(g)$ generates $\overline{G}$ (otherwise $f$ would be trivial on $G$).  Then by Lemma \ref{coboundary}(a), the operator $\pi_2(g)$ fixes a vector of the form $(-1, w_1, w_2, 0) \in \mathcal{T} / 2\mathcal{T}$, implying that all of $\overline{G}$ fixes the subspace of $\mathcal{T} / 2\mathcal{T}$ spanned by $\{(-1, w_1, w_2, 0), (0, 0, 0, 1)\}$ and so $G$ is not a counterexample.

We therefore assume that $\overline{G} \cong \pi(G) = \SL_2(\F_2)$.  Let $w = (w_1, w_2) \in \F_2^2$ be the vector provided by Lemma \ref{coboundary}(b) applied to $G$, and let $S \subset \overline{G}_2 \times \langle x_4 \rangle$ be the subgroup which fixes the vector $(-1, w_1, w_2, 0)$.  A straightforward calculation similar to the ones done in the proof of Lemma \ref{coboundary}(a) shows that the first three entries of $(-1, w_1, w_2, 0)$ are fixed under multiplication by every matrix in $\overline{G} \times \langle x_4 \rangle$; meanwhile, it is immediate that $x_4$ acts by changing the final entry of any vector in $\mathcal{T} / \ell \mathcal{T}$ whose first entry is nontrivial.  It follows that for each $y \in \overline{G}_2$ we have $y \in S$ or $x_4 y \in S$ and so $S \cong \SL_2(\F_2)$.  If $\overline{G}_2 = S$, then $\overline{G}_2$ is not a counterexample by Proposition \ref{latticeprop}(c) because it fixes the subspace of $\mathcal{T} / \ell \mathcal{T}$ spanned by $\{(-1, w_1, w_2, 0), (0, 0, 0, 1)\}$.  The only alternative is that $\overline{G}_2 = \{x_4^{\phi(x)} y \ | \ y \in S\}$ where $\phi$ is the surjection from $\SL_2(\F_2)$ to $\Z / 2$.  In this case, since not everything in $\overline{G}_2$ fixes $(-1, w_1, w_2, 0)$; the unipotent elements already each fix some $2$-dimensional subspace; and the non-unipotent elements do fix the $2$-dimensional subspace spanned by $\{(-1, w_1, w_2, 0), (0, 0, 0, 1)\}$ but cannot fix a larger subspace (as was argued above), we get that $\overline{G}$ does not fix a $2$-dimensional subspace of $\mathcal{T} / \ell \mathcal{T}$.  Therefore, as long as $f$ is not trivial on $G$ (which means that either $f$ is nontrivial on $G \cap \ker(\pi_2)$ or that it factors through the only nontrivial homomorphism from $\overline{G} \cong \SL_2(\F_2)$ to $\Z / 2$), we have that $G$ is a counterexample by Proposition \ref{latticeprop}(c) and by a quick check of quotients of the $G$-stable sublattices provided by Lemma \ref{paraboliclattices}.  Now it is clear as before that for any subgroup $H \subset \ker(\pi_2) \cap \ker(f)$, the group generated by the  subgroups $G$ and $H$ is also a counterexample because multiplying by elements in $H$ will not affect images under $\alpha$, $\beta$, and $\delta$; since each element of $\overline{G}_2$ fixes a $2$-dimensional subgroup of $\F_2^4$, we do not need $f$ to take a certain value on any particular element of $\overline{G}_2$.
\end{proof}

\end{document}